\documentclass[12pt]{amsart}       
\usepackage{txfonts}
\usepackage{amssymb}
\usepackage{eucal}
\usepackage{graphicx}
\usepackage{amssymb}
\usepackage{amsmath}
\usepackage{amscd}
\usepackage[all]{xy}           
\usepackage{amsfonts,latexsym}
\usepackage{xspace}
\usepackage{epsfig}
\usepackage{float}
\usepackage{mathrsfs} 
\usepackage{color}
\usepackage{fancybox}
\usepackage{colordvi}
\usepackage{multicol}
\usepackage{colordvi}
\usepackage{ifpdf}
\usepackage[colorlinks,final,backref=page,hyperindex]{hyperref}

\usepackage[active]{srcltx} 

\makeatletter
\newcommand{\model@english@today}{%
  \ifcase\month\or January\or February\or March\or April\or May\or June%
  \or July\or August\or September\or October\or November\or December\fi
  \space\number\day, \number\year}
\let\today\model@english@today
\AtBeginDocument{%
  \let\today\model@english@today
}
\makeatother

\usepackage{amsmath, amssymb, mathtools}
\usepackage{tikz-cd} 
\usepackage{tikz}

\usepackage{graphicx}

\usepackage{longtable}
\usepackage{booktabs}
\usepackage{array}

\usepackage{xcolor}
\usepackage[most]{tcolorbox}
\usepackage{enumitem}
\newtheorem{theorem}{Theorem}[section]

\newtheorem{proposition}[theorem]{Proposition}
\newtheorem{lemma}[theorem]{Lemma}
\newtheorem{coro}[theorem]{Corollary}
\newtheorem{corollary}[theorem]{Corollary}
\newtheorem{prop-def}{Proposition-Definition}[section]
\newtheorem{coro-def}{Corollary-Definition}[section]

\theoremstyle{definition}
\newtheorem{definition}[theorem]{Definition}
\newtheorem{remark}[theorem]{Remark}

\newtheorem{assumption}[theorem]{Assumption}

\newcommand{\nc}{\newcommand}
\nc{\Liu}[1]{\textcolor{red}{#1}}
\nc{\tblue}[1]{\textcolor{blue}{#1}}
\nc{\tgreen}[1]{\textcolor{green}{#1}}
\nc{\tpurple}[1]{\textcolor{purple}{#1}}
\nc{\btred}[1]{\textcolor{red}{\bf #1}}
\nc{\btblue}[1]{\textcolor{blue}{\bf #1}}
\nc{\btgreen}[1]{\textcolor{green}{\bf #1}}
\nc{\btpurple}[1]{\textcolor{purple}{\bf #1}}
\nc{\NN}{{\mathbb N}}
\nc{\ncsha}{{\mbox{\cyr X}^{\mathrm NC}}} \nc{\ncshao}{{\mbox{\cyr
X}^{\mathrm NC}_0}}

\newcommand{\delete}[1]{}

\nc{\mlabel}[1]{\label{#1}}
\nc{\mcite}[1]{\cite{#1}}
\nc{\mref}[1]{\ref{#1}}
\nc{\meqref}[1]{\eqref{#1}}
\nc{\mbibitem}[1]{\bibitem{#1}}

\delete{
\nc{\mlabel}[1]{\label{#1}{\hfill \hspace{1cm}{\bf{{\ }\hfill(#1)}}}}
\nc{\mcite}[1]{\cite{#1}{{\bf{{\ }(#1)}}}}
\nc{\mref}[1]{\ref{#1}{{\bf{{\ }(#1)}}}}
\nc{\meqref}[1]{\eqref{#1}{{\bf{{\ }(#1)}}}}
\nc{\mbibitem}[1]{\bibitem[\bf #1]{#1}}
}
\font\cyr=wncyr10 
\font\scyr=wncyr6
\nc{\sha}{{\mbox{\scyr X}}}
\nc{\shap}{{\mbox{\cyrs X}}} 
\nc{\shpr}{\diamond}    
\nc{\shp}{\ast} \nc{\shplus}{\shpr^+}
\nc{\shprc}{\shpr_c}    
\nc{\dep}{\mrm{dep}} \nc{\lc}{\lfloor} \nc{\rc}{\rfloor}
\nc{\db}{\leq_{\rm db}} \nc{\bfk}{\bf k} \nc{\RR}{\mathbb{R}}

\font\cyr=wncyr10 \font\cyrs=wncyr7
\nc{\li}[1]{\textcolor{red}{#1}}
\nc{\lir}[1]{\textcolor{red}{Li:#1}}
\nc{\yi}[1]{\textcolor{blue}{Yi: #1}}
\nc{\xing}[1]{\textcolor{purple}{Xing:#1}}
\nc{\revise}[1]{\textcolor{red}{#1}}
\nc{\q}[1]{\textcolor{blue}{q: #1}}

\providecommand{\dt[1]}{\Delta_{[#1]}}   \providecommand{\loc}{{\rm loc}}

\NewDocumentCommand{\sint}{e{_^}}{\mathop{\textstyle\int}\nolimits_{\mkern-6mu #1}^{#2}}
 
\newtcolorbox{qshuStepBox}[1]{
  enhanced,
  breakable,
  colback=blue!2,
  colframe=blue!45!black,
  title={#1},
  fonttitle=\bfseries,
  boxrule=0.45pt,
  arc=1mm,
  left=1.2em,
  right=1.2em,
  top=0.7em,
  bottom=0.7em,
  before skip=0.9em,
  after skip=0.9em
}

\begin{document}

\title[Random dynamical systems generated by drift-containing CRDE]{Random dynamical systems generated by rough differential equations driven by controlled rough paths}

%
\author{Xing Gao$^*$}\thanks{*Corresponding author}
\address{School of Mathematics and Statistics, Lanzhou University,
Lanzhou, 730000, China; Gansu Provincial Research Center for Basic Disciplines of Mathematics and Statistics, Lanzhou, 730070, China
}
\email{gaoxing@lzu.edu.cn}

\author{Jiaqi Liu}
\address{School of Mathematics and Statistics, Lanzhou University
Lanzhou, 730000, China
}
\email{jiaqiliu2025@lzu.edu.cn}

\date{\today}
\begin{abstract}
We study rough differential equations driven by controlled rough paths with drift. By means of the canonical lift of the controlled driver, we establish pathwise well-posedness and continuity. Under compatible time-shift conditions, we construct differentiable solution cocycles for the associated dynamical system generated by the equations considered here. Furthermore, assuming the corresponding regularity, moment, and spectral hypotheses, we prove the invariance of the local stable, unstable, and center manifolds.
\end{abstract}

\makeatletter
\@namedef{subjclassname@2020}{\textup{2020} Mathematics Subject Classification}
\makeatother
\subjclass[2020]{
60L20, 
60L99, 
37H10, 
37H15, 
37H30. 
}

\keywords{controlled rough paths, rough differential equations, random dynamical systems, invariant manifolds.}

\maketitle

\tableofcontents

\setcounter{section}{0}

\allowdisplaybreaks

\section{Introduction}\label{sec:introduction}
We first explain the passage from classical rough differential equations (RDEs)
to rough differential equations driven by controlled rough paths (CRDEs), and
then study the solution cocycles generated by these equations and the
associated random dynamical systems. We finally summarize the main
challenges arising in this dynamical framework and the corresponding
contributions of the present paper.

\subsection{From RDE to CRDE}
\label{subsec:intro-rde-to-crde}

RDEs provide a pathwise approach to
differential equations driven by irregular signals.
Building on Young's integration theory~\cite{Young1936},
Lyons~\cite{Lyons1998} developed rough path theory by enhancing the
driving signal with iterated-integral information.
This makes the solution map well defined and continuous below the
classical Young regularity threshold and provides a deterministic
framework for Wong-Zakai type approximations~\cite{WZ1965}.
Further developments can be found in~\cite{FH20, FV10b, LQ2002}.
In the level-two regime considered here, the input is a rough path
${\bf X}=(X,\mathbb X)$, and a classical RDE with drift has the form
\begin{equation*}
\label{eq:intro-classical-rde}
dY_t
=
F(Y_t)\,d{\bf X}_t
+
F_0(Y_t)\,dt.
\end{equation*}
which is introduced explicitly as~\eqref{eq-diff-RDE} in
Section~\ref{sec:preli}.
The second level $\mathbb X$ provides the iterated-integral information
needed to interpret the rough integral.

Gubinelli's theory of controlled rough paths~\cite{Gubinelli04}
provides a complementary viewpoint.
For $\alpha\in(1/3,1/2]$, an ${\bf X}$-controlled path
${\bf Z}=(Z,Z')$ satisfies
\[
Z_{s,t}
=
Z'_sX_{s,t}
+
R^Z_{s,t},
\qquad
\|R^Z_{s,t}\|
\lesssim
|t-s|^{2\alpha},
\]
where $Z'$ describes the leading response of $Z$ to the reference
rough path $X$.
This framework also allows one controlled rough path to be integrated
against another~\cite{Gubinelli04}; quantitative continuity estimates
were further developed in~\cite{BG2022}.

The distinction between the reference signal and the effective driver
is particularly important in multi-layer rough systems, where the
signal entering an upper-level equation may itself be generated by a
lower-level rough equation.
A basic motivating example is
\begin{equation}
\label{eq:intro-layered-system}
\begin{aligned}
dZ_t
=
G(Z_t)\,d{\bf X}_t
+
G_0(Z_t)\,dt,
\qquad 
dY_t
=
F(Y_t)\,d{\bf Z}_t
+
F_0(Y_t)\,dt.
\end{aligned}
\end{equation}
Under suitable regularity assumptions, the first equation produces an
${\bf X}$-controlled path
\[
{\bf Z}=(Z,Z'),
\qquad
Z'_t=G(Z_t),
\]
which then drives the second equation.
Such systems naturally lead to rough differential equations driven by
controlled rough paths (CRDEs); see~\cite{FH20,FV10b,Gubinelli04}.
The pathwise theory of driftless CRDEs was developed in~\cite{LiGao2025},
where the controlled-against-controlled rough integral was revisited
and a universal limit theorem for CRDEs was established, extending
the classical universal limit theorem for RDEs to the controlled-driver setting.
From a numerical perspective, controlled B-series and controlled
Runge-Kutta methods for driftless CRDEs were recently developed
in~\cite{AnFanGao2026}, together with tree-based order conditions and
corresponding convergence estimates.

In this paper, we study the drift-containing CRDE
\begin{equation}
\label{eq:intro-controlled-rde}
dY_t
=
F(Y_t)\,d{\bf Z}_t
+
F_0(Y_t)\,dt,
\end{equation}
with ${\bf Z}$ an ${\bf X}$-controlled rough path.
A precise formulation is given in~\eqref{main-diff-CRDE} in
Section~\ref{sec:crde-theory}.
A solution is represented by the controlled pair
\[
{\bf Y}=(Y,Y'),
\qquad
Y'_t=F(Y_t)Z'_t.
\]
The classical RDE is recovered by taking $U=V$ and
${\bf Z}=(X,\mathrm{Id}_V)$.
There are three main reasons for studying drift-containing CRDEs
directly.
First, in the layered system~\eqref{eq:intro-layered-system}, the
natural structure is
\[
{\bf X}
\longrightarrow
{\bf Z}
\longrightarrow
{\bf Y}.
\]
The controlled formulation keeps visible how the effective driver
${\bf Z}$ is generated from ${\bf X}$ and how perturbations propagate
through the system.

Second, the drift terms are not merely auxiliary.
The term $G_0$ affects the evolution of the effective driver, while
$F_0$ contributes independently to the downstream dynamics.
For example, if the controlled rough driver is constant, then
\eqref{eq:intro-controlled-rde} reduces to
\[
dY_t=F_0(Y_t)\,dt,
\]
whereas the corresponding driftless equation has only constant
trajectories.
Thus the drift may change equilibria, linearization, stability, and
the local dynamical geometry, which is essential for our later study
of stationary points and invariant manifolds.

Finally, an ${\bf X}$-controlled driver ${\bf Z}=(Z,Z')$ does not
come with an a priori second level above $Z$.
Starting from the controlled data $({\bf X},{\bf Z})$, one constructs
\[
\mathbb Z_{s,t}
:=
\sint_s^t Z_{s,r}\otimes d Z_r,
\qquad
\mathcal Z=(Z,\mathbb Z);
\]
see~\cite{LiGao2025}.
This construction is recalled in
Lemma~\ref{thm:lift-canonical-lift-rough-path}.
Theorem~\ref{thm:solution-equivalence} shows that a CRDE solution
$(Y,F(Y)Z')$ corresponds to an RDE solution $(Y,F(Y))$ driven by
$\mathcal Z$. Thus the two formulations have the same state path $Y$,
but different controlled solution structures.
Nevertheless, the controlled formulation remains useful because the
pair $({\bf X},{\bf Z})$ retains how the effective driver is generated
from the underlying rough signal.
This additional structure is crucial for layered rough systems and
for the random dynamical analysis developed below; see also
Remark~\ref{rmk:controlled-formulation-significance}.

\subsection{Solution cocycles and random dynamics}
\label{subsec:intro-solution-cocycles}

Finite-time well-posedness describes individual solution paths,
but does not by itself describe their long-time organization.
Random dynamical systems (RDS) provide the appropriate framework
for studying invariant objects and stability along a fixed
random environment~\cite{Arnold1998,ArnoldScheutzow1995,
CF1994,Kifer1986}.
Their defining consistency condition is the cocycle identity:
\[
\varphi_{t+s}^{\omega}
=
\varphi_t^{\theta_s\omega}\circ\varphi_s^\omega,
\qquad s,t\ge0.
\]
It expresses the compatibility between restarting a solution
and shifting the underlying environment.
For RDEs, this compatibility must hold for the enhanced driving
signal, not merely for its first-level path.
Bailleul, Riedel, and Scheutzow~\cite{BRS2017} established
sufficient conditions under which stochastic rough path lifts
satisfy the cocycle property and the corresponding RDEs
generate RDS.
Their results include suitable lifts of fractional Brownian
motion.
This pathwise construction makes dynamical methods available
without requiring the state process itself to possess
the Markov property.

Local random dynamics are naturally studied along a stationary
trajectory.
A stationary point $Y_\omega$ satisfies
$
\varphi_t^\omega(Y_\omega)=Y_{\theta_t\omega}.
$
When the state maps are differentiable, their derivatives
along this trajectory form a linear cocycle.
Under the hypotheses of the multiplicative ergodic theorem, the linearized dynamics admit Lyapunov exponents and
an Oseledets decomposition~\cite{Oseledets1968}.
The connection between these growth rates and nonlinear
invariant manifolds is a fundamental theme of smooth ergodic
theory~\cite{Ruelle1979}.
In the stochastic setting, Mohammed and
Scheutzow~\cite{MS1999} established local stable and unstable
manifolds near hyperbolic stationary trajectories of
stochastic differential equations.
For RDE-generated systems, Ghani Varzaneh and
Riedel~\cite{GhaniVarzanehRiedel2025} proved local stable,
unstable, and center manifold results around random fixed
points, together with exponential stability when the largest
Lyapunov exponent is negative.
Their work shows how rough path estimates can supply the
analytic inputs needed for local random dynamical arguments.
The abstract results on Oseledets splittings and invariant
manifolds in~\cite{GVR2023}, and the center manifold theorem
in~\cite{center2025}, provide related tools in the setting
of measurable fields of Banach spaces.
These developments suggest a corresponding program for CRDEs:
starting from a random reference rough path and a random
controlled driver, construct a solution cocycle and determine
the local geometry near a stationary trajectory.
The issue is not simply whether each CRDE can be written as
an RDE, but whether the controlled data supply the compatibility
and quantitative estimates required by this dynamical theory.

\subsection{Challenges and contributions in CRDE dynamics}
\label{subsec:Diff}

The passage from RDE dynamics to CRDE dynamics involves two
distinct steps.
One must first assemble the pathwise solutions into a cocycle,
and then control its linear and nonlinear behavior over
arbitrarily many time intervals.
The canonical lift connects the equations at the level
of individual trajectories, but this connection alone does
not verify the hypotheses needed for either step.
Our contribution is to develop these additional inputs in
terms of the original controlled data.

The first requirement is compatibility with time shifts.
Being controlled by ${\bf X}$ is a pathwise regularity
condition; it does not imply that ${\bf Z}$ has the cocycle
structure required for random dynamics.
In particular, a nonlinear transformation or an upstream
solution need not inherit the necessary shift identities
automatically.
Definition~\ref{def:prelim-rough-path-cocycle} specifies the
compatibility conditions for both the increments of $Z$
and its Gubinelli derivative:
\[
Z_{s,s+t}(\omega)
=
Z_{0,t}(\theta_s\omega),
\qquad
Z'_{s+t}(\omega)
=
Z'_t(\theta_s\omega).
\]
Together with the rough path cocycle identities for
${\bf X}$, these conditions also control the shift of the
compensated terms defining the rough integral.

Lemma~\ref{lemma-def-tilde-Y} establishes the resulting shift
identity for the controlled integral and the drift integral.
This identity and pathwise uniqueness lead to the solution
cocycle in Theorem~\ref{thm:solution-cocycle}.
The continuity of the finite-time solution construction
provides the route from measurable driving data to measurable
state maps.
Under the stronger smoothness assumptions used for
differentiation,
Proposition~\ref{thm:solution-Ck-regularity} identifies the
variational equation, and
Corollary~\ref{cor:solution-Ck-cocycle} gives the corresponding
$C^k$-cocycle.
These results make the restart structure of a CRDE compatible
with the underlying random environment.

The second requirement is quantitative control of the
linearized evolution.
Continuity on a fixed interval does not provide the
logarithmic integrability required by the multiplicative ergodic theorem.
Moreover, the effective driver $Z$ need not retain the
probabilistic structure of the original noise.
For instance, Gaussian rough path estimates cannot be
applied to an arbitrary controlled output merely because
the reference signal is Gaussian.
The importance of this distinction is illustrated by the
Jacobian integrability and tail estimates for Gaussian
RDEs obtained in~\cite{CLL2013}.

We therefore estimate the derivative and its inverse directly
in terms of the reference rough path, the controlled driver,
and the initial state.
The deterministic bounds in
Propositions~\ref{prop:lin-deterministic-derivative-bound}
and~\ref{prop:lin-deterministic-inverse-bound} are used,
under the stated moment assumptions, to obtain the
logarithmic integrability of the linearized cocycle in
Corollary~\ref{cor:lin-local-linear-control-integrable}.
The local theory is developed around a prescribed stationary
point with the required integrability.

The third requirement concerns the nonlinear remainder.
The Lyapunov exponents describe only the linearized equation;
invariant manifolds require the nonlinear terms to remain
sufficiently small relative to this linear behavior.
For a fixed time step $t_0>0$, consider
\[
P_\omega(z)
=
\phi_{t_0}^\omega(Y_\omega+z)
-
\phi_{t_0}^\omega(Y_\omega)
-
D_{Y_\omega}\phi_{t_0}^\omega[z].
\]
The relevant estimate is a local bound of the form
\[
\|P_\omega(z_2)-P_\omega(z_1)\|
\le
f_{\mathrm{nl}}(\omega)
(\|z_1\|+\|z_2\|)
\|z_2-z_1\|.
\]
Finiteness of the random coefficient is not enough:
its growth under repeated shifts must also be controlled.
Proposition~\ref{prop:manifold-derivative-lipschitz}
estimates the local variation of the derivative, while
Proposition~\ref{lem:manifold-fnl-tempered} provides the
temperedness used in the nonlinear remainder argument.
This is the quantitative connection between the finite-time
controlled calculus and the local invariant manifold theory.

Within the applicable multiplicative ergodic theorem setting,
Lemma~\ref{lemma-29} supplies the Oseledets decomposition
for the linearized cocycle.
Theorems~\ref{thm:manifold-local-stable}
and~\ref{thm:manifold-local-unstable} then describe the local
nonlinear geometry associated with the negative and positive
Lyapunov directions.
The stable manifold consists of nearby initial states whose
forward trajectories track the stationary trajectory,
whereas the unstable manifold is characterized by backward
orbits approaching that trajectory in the past.
Theorem~\ref{thm:manifold-local-center} treats the zero
Lyapunov directions through a localized cocycle.
The abstract manifold results are those
of~\cite{GVR2023,center2025}; the work specific to CRDEs
is the verification of their analytic inputs from the
controlled-driver structure.

Finally, the manifold construction is initially performed
for the time-$t_0$ cocycle.
Behavior at the sampling times alone does not control
the evolution between them.
Proposition~\ref{prop:manifold-tempered-finite-lipschitz}
provides a tempered finite-time Lipschitz bound on
$[0,t_0]$, which is the key input for the continuous-time
analysis in
Subsection~\ref{subsec:manifold-continuous-time}.
The need for this estimate is distinct from the
finite-time continuity of the solution map: it controls
the random amplification factors encountered when
successive time intervals are concatenated.

The significance of the resulting framework is therefore
structural and quantitative.
It connects local random dynamical theory with models whose
effective driver is specified through a controlled expansion,
rather than through an independently prescribed rough path
lift.
The equivalence with classical RDEs is used as an analytic
tool, while the reference signal, the controlled derivative,
and the remainder remain visible in the estimates.
This makes it possible to formulate local dynamical
questions directly in the variables supplied by a layered
rough model.

\smallskip\noindent
{\bf Outline of the paper.}
Section~\ref{sec:preli} recalls rough paths, controlled paths,
and RDS, and specifies the controlled rough path cocycle
structure in
Definition~\ref{def:prelim-rough-path-cocycle}.
Section~\ref{sec:crde-theory} establishes the equivalence
with classical RDEs in
Theorem~\ref{thm:solution-equivalence}, the finite-time
well-posedness and continuity result in
Theorem~\ref{thm:solution-pathwise-wellposedness}, and
the solution cocycle in
Theorem~\ref{thm:solution-cocycle}.
Section~\ref{sec:stationary-points} develops the derivative
and inverse estimates in
Propositions~\ref{prop:lin-deterministic-derivative-bound}
and~\ref{prop:lin-deterministic-inverse-bound}, together
with the tempered nonlinear control in
Proposition~\ref{lem:manifold-fnl-tempered}.
Section~\ref{sec:local-invariant-manifolds} presents the
Oseledets decomposition in Lemma~\ref{lemma-29} and the
local stable, unstable, and center manifold results in
Theorems~\ref{thm:manifold-local-stable},
\ref{thm:manifold-local-unstable},
and~\ref{thm:manifold-local-center}.
It concludes with the finite-time Lipschitz estimate in
Proposition~\ref{prop:manifold-tempered-finite-lipschitz}
and the continuous-time analysis in
Corollary~\ref{cor:manifold-continuous-time-stable}
and Theorem~\ref{thm-continu}.

\smallskip\noindent
\textbf{Notation.}
We use $\mathbb N$, $\mathbb N_0:=\{0\}\cup\mathbb N$, and $\mathbb R$
for the sets of positive integers, nonnegative integers, and real numbers,
respectively. All spaces are Banach spaces, and $\mathcal L(V,W)$ denotes
the space of bounded linear maps from $V$ to $W$. When no confusion can
arise, we write $\|\cdot\|$ for the relevant Banach-space norm. For
$k\in\mathbb N_0$, let $C^k(V,W)$ denote the space of $k$-times
continuously differentiable maps from $V$ to $W$, and let
$C_b^k(V,W)$ be the subspace whose derivatives up to order $k$ are
bounded. 
For a continuous path $X:[0,T]\to V$, we denote its increment over
$[s,t]$ by $X_{s,t}:=X_t-X_s$, and set
$
\dt{0,T}:=\{(s,t):0\le s\le t\le T\}.
$

\section{Preliminaries}\mlabel{sec:preli}
In this section, we recall the basic notions of rough paths, controlled rough paths, and random dynamical systems needed in the sequel. We then combine these two frameworks through rough path cocycles and controlled rough path cocycles, which will be used to construct the solution cocycle for the drift-containing CRDE.

\subsection{Rough paths and controlled rough paths}\mlabel{subsec:prelim-rough-paths}
We begin by recalling the level-two $\alpha$-H\"older rough path framework and the associated notion of controlled rough paths for $\alpha\in(1/3,1/2]$, together with the basic norms and remainder structure that will be used throughout the paper.

\begin{definition}\label{def:prelim-level-two-rough-path}~\cite{BG2022,FH20,GhaniVarzanehRiedel2025}
Let $(V, \|\cdot \|)$ be a finite-dimensional Banach space and let $\alpha\in(1/3,1/2]$.
An {\bf $\alpha$-H\"older rough path} over $V$ on $[0,T]$ is a pair $\mathbf X=(X,\mathbb X)$, where
\[
X:[0,T] \to V,
\qquad
\mathbb X: \dt{0,T} \to V\otimes V,
\]
such that
\begin{enumerate}
\item $X$ is $\alpha$-H\"older continuous, and $\mathbb X$ is $2\alpha$-H\"older continuous, that is
\[\|X\|_{\alpha,[0,T]}:=\sup_{s\neq t\in [0,T]}\frac{\|X_{s,t}\|}{|t-s|^{\alpha}}<\infty,\qquad \|\mathbb X\|_{2\alpha,[0,T]}:=\sup_{s\neq t\in [0,T]}\frac{\|\mathbb X_{s,t}\|}{|t-s|^{2\alpha}}<\infty;\]
\item Chen's relation holds:
\begin{equation*}
\mathbb X_{s,t}
=
\mathbb X_{s,u}
+
\mathbb X_{u,t}
+
X_{s,u}\otimes X_{u,t},
\qquad
0\le s\le u\le t\le T.
\end{equation*}
\end{enumerate}
\end{definition}

We denote the space of all such $\alpha$-H\"older rough paths by
$\mathcal D^\alpha([0,T],V)$. Following~\mcite{GhaniVarzanehRiedel2025}, we equip this space with the rough path norm
\begin{equation*}
\|\mathbf X\|_{\alpha,[0,T]}
:=\max\Big\{
\|X\|_{\alpha,[0,T]},
\|\mathbb X\|_{2\alpha,[0,T]}^{1/2} \Big\}.
\end{equation*}

\begin{definition}\label{def-controlledRP}~\cite{BG2022,FH20, GhaniVarzanehRiedel2025}
Let  $(W, \|\cdot \|)$ be a finite-dimensional Banach space, and let $\mathbf X=(X,\mathbb X)\in \mathcal D^\alpha([0,T],V)$. A pair $\mathbf Y=(Y,Y')$ is called an {\bf ${\bf X}$-controlled rough path} with 
\[
Y: [0,T]\to W,
\qquad
Y': [0,T]\to \mathcal L(V,W),
\]
if 
\[
\|R^{Y}\|_{2\alpha,[0,T]}
:=
\sup_{0\le s<t\le T}\frac{\|R^{Y}_{s,t}\|}{|t-s|^{2\alpha}}<\infty,  \qquad  \|Y'\|_{\alpha,[0,T]}
:=
\sup_{0\le s<t\le T}\frac{\|Y'_{s,t}\|}{|t-s|^{\alpha}}<\infty.
\]
Here the remainder $R^{Y}:\Delta_{[0,T]} \rightarrow W$ is given by 
\begin{equation}
	R^{Y}_{s,t}:=Y_{s,t}-Y'_sX_{s,t}, \qquad \forall  0\le s< t\le T. \mlabel{eq:remainder}
\end{equation}
\end{definition}

We denote the space of all such ${\bf X}$-controlled rough paths by
$\mathcal C_{\bf X}^\alpha([0,T],W)$. Following~\cite{GhaniVarzanehRiedel2025}, we equip this space with the norm
\[
\|\mathbf Y\|_{\mathcal{C}_{\bf X}^\alpha([0,T],W)}
:=
\max\Big\{
\|Y\|_{\infty,[0,T]},
\|Y'\|_{\infty,[0,T]},
\|R^{Y}\|_{2\alpha,[0,T]},
\|Y'\|_{\alpha,[0,T]}
\Big\},
\]
for $\mathbf Y=(Y,Y')\in\mathcal C_{\bf X}^\alpha([0,T],W)$.
When the state space $W$ is clear from the context, we simply write
$
\|\cdot\|_{\mathcal C_{\bf X}^\alpha([0,T])}
$
instead of
$
\|\cdot\|_{\mathcal C_{\bf X}^\alpha([0,T],W)}.
$

For continuity estimates, it is also important to quantify the distance between controlled rough paths defined with respect to different driving rough 
paths. Let 
$$\mathbf X,\,\widetilde{\bf X}\in \mathcal D^\alpha ([0,T],V), \qquad \mathbf Y=(Y,Y')\in\mathcal C_{\bf X}^\alpha([0,T],W), \qquad \widetilde{\mathbf Y}=(\widetilde{Y},\widetilde{Y'})\in\mathcal C_{\widetilde{\bf X}}^\alpha([0,T],W).$$ 
Define the distance~\cite[p.~7]{BG2022}
$$
d_{\mathbf X,\widetilde{\mathbf X};\alpha}(\mathbf Y,\widetilde{\mathbf{Y}}):=\|R^Y-R^{\widetilde{Y}}\|_{2\alpha,[0,T]}+\|Y'-\widetilde{Y'}\|_{\alpha,[0,T]}.
$$
We will also use the controlled rough path framework to formulate rough differential equations with drift. In particular, consider
\begin{align}
dY_t
=
F(Y_t)\,d{\bf X}_t
+
F_0(Y_t)\,dt,
\mlabel{eq-diff-RDE}
\end{align}
with initial value $Y_0\in W$, where
$F:W\to\mathcal L(V,W)$ and $F_0:W\to W$.
Its integral formulation is
\begin{align*}
Y_t
=
Y_0
+
\sint_0^t F(Y_r)\,d{\bf X}_r
+
\sint_0^t F_0(Y_r)\,dr,
\end{align*}
and the equation is solved in the space of ${\bf X}$-controlled rough paths.

\subsection{Random dynamical systems}
We next recall the basic framework of random dynamical systems that will be used to describe the long-time dynamics of the solutions, with particular emphasis on measurable cocycles, stationary points, and their linearizations.

\begin{definition}~\cite{Walter1982}
\label{def:prelim-metric-dynamical-system}
An {\bf invertible measure-preserving dynamical system} is a quadruple $(\Omega,\mathcal F,\\\mathbb P,\{\theta_t\}_{t\in\mathbb R})$, where $(\Omega,\mathcal F,\mathbb P)$ is a probability space and $\{\theta_t\}_{t\in\mathbb R}$ is a family of measurable maps $\theta_t:\Omega\to \Omega$ satisfying

\begin{enumerate}
\item $\theta_0=\mathrm{id}_{\Omega}$;

\item $\theta_{t+s}=\theta_t\circ\theta_s$ for all $s,t\in\mathbb R$;

\item the map $\theta : \mathbb{R}\times \Omega\to \Omega, (t,\omega)\mapsto\theta_t\omega$ is $(\mathcal B(\mathbb R)\otimes\mathcal F,\mathcal F)$-measurable; 

\item $\mathbb P\circ\theta_t^{-1}=\mathbb P$ for all $t\in\mathbb R$.
\end{enumerate}
Moreover, if every $\theta_t$ with $t\ne0$ is ergodic, the system is called {\bf ergodic}.
\end{definition}

\begin{definition}\label{def:prelim-measurable-cocycle}~\cite{BRS2017,Mackey1966}
Let $E$ be a finite-dimensional Banach space.
\begin{enumerate}
\item~\label{it-def:prelim-measurable-cocycle} A measurable map
\begin{equation}
\varphi:\mathbb R \times\Omega\times E\to E
\mlabel{eq:mecoc0}
\end{equation}
is called a {\bf measurable cocycle} over $\{\theta_t\}_{t\in\mathbb{R}}$ if
\begin{equation}
\varphi(t+s,\omega,x)
= \varphi\bigl(t,\theta_s\omega,\varphi(s,\omega,x)\bigr), \qquad \forall s,t\in \mathbb{R},\,\omega\in \Omega,\,x\in E.
\mlabel{eq:mecoc}
\end{equation}

\item~\label{it-def:prelim-Ck-cocycle} Moreover, for fixed nonnegative integer $k$ and $(t,\omega)\in \mathbb{R}\times \Omega$, if $$\varphi(t,\omega,\cdot):E\to E, \qquad x\mapsto \varphi(t,\omega,x)$$ belongs to $C^k(E,E)$, then $\varphi$ is called a {\bf $C^k$-cocycle}. 

\item A measurable cocycle $\psi$ is called a {\bf linear cocycle on $E$} if $\psi(t,\omega,\cdot)\in \mathcal L(E,E) $ for every $(t,\omega)\in \mathbb{R}\times \Omega$. 
\end{enumerate}
\end{definition}

The measurable cocycle in~\meqref{eq:mecoc0} induces a map
\[
\varphi_t^\omega: E\rightarrow E, \qquad  x\mapsto \varphi(t,\omega, x),
\]
and the measurable cocycle property~\meqref{eq:mecoc} becomes
$
\varphi_{t+s}^\omega
=
\varphi_t^{\theta_s\omega}\circ\varphi_s^\omega.
$

\begin{definition}~\cite{Arnold1998}
\label{def:prelim-stationary-point}
Let $\varphi:\mathbb R \times\Omega\times E\to E$ be a measurable cocycle. 
A measurable map $$Y:\Omega\to E,\qquad \omega\mapsto Y_\omega$$ is called a {\bf stationary point} of the cocycle $\varphi$ if
\begin{equation}~\label{eq:stationary}
\varphi(t,\omega,Y_\omega)=Y_{\theta_t\omega},
\qquad
\forall t\ge0,\,\omega\in\Omega.
\end{equation}
\end{definition}

\begin{remark}~\cite{GhaniVarzanehRiedel2025}
Let
$
\varphi:\mathbb R\times\Omega\times E\to E
$
be a $C^1$-cocycle, and let $Y:\Omega\to E$ be a stationary point of $\varphi$.
Define
\[
\psi:\mathbb R\times\Omega\times E\to E,
\qquad
(t,\omega,x)\mapsto
\psi(t,\omega,x)
:=
D_{Y_\omega}\varphi(t,\omega,\cdot)[x].
\]
Then $\psi$ is a linear cocycle, where
$D_{Y_\omega}\varphi(t,\omega,\cdot)\in\mathcal L(E,E)$
denotes the Fr\'echet derivative of the map
$x\mapsto\varphi(t,\omega,x)$ at $Y_\omega$.
\end{remark}

\subsection{Rough path cocycles and controlled rough path cocycles}
We now combine the rough path framework with the random dynamical systems viewpoint by recalling rough path cocycles and introducing controlled rough path cocycles. We also record the corresponding shift property, which will be used later in constructing the solution cocycle.

\begin{definition} \label{def:prelim-rough-path-cocycle}
Let $\alpha\in(1/3,1/2]$, and let $(\Omega,\mathcal F,\mathbb P,\{\theta_t\}_{t\in\mathbb R})$ be an invertible measure-preserving dynamical system. 
\begin{enumerate}
\item~\cite{BRS2017} 
A (random) process $\mathbf X=(X,\mathbb X)$ with 
$$X:\mathbb R\times \Omega\to V,\qquad \mathbb X:\mathbb R\times \mathbb R\times \Omega\to V\otimes V$$ 
is called a (random) {\bf rough path cocycle} if  $$\mathbf X(\omega)= (X(\omega),\mathbb X(\omega))\in\mathcal{D}^{\alpha}(\mathbb R,V), \qquad \forall \omega\in \Omega,$$   
and if the following cocycle relation holds:
\begin{align}~\label{eq:cocycle-relation}
X_{s,s+t}(\omega)=X_{0,t}(\theta_s\omega),
\qquad
\mathbb X_{s,s+t}(\omega)=\mathbb X_{0,t}(\theta_s\omega), \qquad \forall s\in\mathbb{R},\,t\in \mathbb{R}_{\ge 0}.
\end{align}

\item Further, a (random) process $\mathbf Y=(Y,Y')$ with 
$$Y:\mathbb R\times \Omega\to W,\qquad Y':\mathbb R\times \Omega\to \mathcal L(V,W)$$
is called a (random) {\bf $\mathbf X$-controlled rough path cocycle} if 
$$\mathbf Y(\omega) = (Y(\omega), Y'(\omega))\in \mathcal C_{\mathbf X(\omega)}^\alpha(\mathbb R,W), \qquad \forall \omega\in \Omega$$ and if the following cocycle relation holds:
\begin{equation}
Y_{s,s+t}(\omega)
=
Y_{0,t}(\theta_s\omega),
 \qquad
Y'_{s+t}(\omega)
=
Y'_t(\theta_s\omega), \qquad \forall s\in\mathbb R,\,t\in\mathbb R_{\ge 0},\,\omega\in \Omega
~\mlabel{shift-eq}
\end{equation}
\end{enumerate}
\end{definition}

Now we study the shift behavior of the remainder term.

\begin{proposition}
\label{prop:lift-remainder-shift}
Let $\mathbf Z =(Z, Z')$ be an $\mathbf X$-controlled rough path cocycle. Then the remainder given in~\meqref{eq:remainder} satisfies 
\begin{equation*} 
R^Z_{s,s+t}(\omega)=R^Z_{0,t}(\theta_s\omega), \qquad \forall s\in\mathbb R,\,t\in\mathbb R_{\ge 0}.
\end{equation*}
\end{proposition}

\begin{proof}
Using the shift relations for $Z$, $Z'$  and $X$,  
\[
Z_{s,s+t}(\omega)
=
Z_{0,t}(\theta_s\omega),\quad
Z'_s(\omega)
=
Z'_0(\theta_s\omega),
\quad
X_{s,s+t}(\omega)
=
X_{0,t}(\theta_s\omega),
\]
which implies that   
\begin{align*}
R^Z_{s,s+t}(\omega)=& \ Z_{s,s+t}(\omega)-Z_{s}'(\omega)X_{s,s+t}(\omega) 
= Z_{0,t}(\theta_{s}\omega)-Z_{0}'(\theta_{s}\omega)X_{0,t}(\theta_{s}\omega) 
=  R_{0,t}^{Z}(\theta_{s}\omega).
\end{align*}
This completes the proof.
\end{proof}

Next, we consider the cocycle property in the following special case.

\begin{proposition}
Let $(\Omega,\mathcal F,\mathbb P,\{\theta_t\}_{t\in\mathbb R})$ be an invertible measure-preserving dynamical system, and let $\mathbf X$ be a rough path cocycle. If $$\mathbf Z(\omega)=(Z(\omega):=X(\omega),Z'(\omega):=\mathrm{Id}_{V})\in\mathcal C_{\bf X(\omega)}^\alpha([0,T],U)\,\text{ for each }\, \omega\in \Omega,$$ then $\mathbf Z$ is an ${\bf X}$-controlled rough path cocycle.
\end{proposition}

\begin{proof}
By Definition~\ref{def:prelim-rough-path-cocycle}, it remains only to verify the shift relations in~\eqref{shift-eq}.
Since ${\bf X}$ is a rough path cocycle,
\[
X_{s,s+t}(\omega)=X_{0,t}(\theta_s\omega),
\qquad
\forall s\in\mathbb R,\ t\in\mathbb R_{\ge0},\ \omega\in\Omega.
\]
Because $Z(\omega)=X(\omega)$, we immediately obtain
\[
Z_{s,s+t}(\omega)=Z_{0,t}(\theta_s\omega).
\]
Moreover, since  
\[
Z'_{s+t}(\omega)
=
\mathrm{Id}_V
=
Z'_t(\theta_s\omega),
\]
${\bf Z}$ satisfies~\eqref{shift-eq} and is therefore an ${\bf X}$-controlled rough path cocycle.
\end{proof}

\section{Rough differential equations driven by controlled rough paths}\mlabel{sec:crde-theory}
In this section, we study drift-containing RDEs driven by controlled rough paths. More precisely, we consider
\begin{equation}\label{main-diff-CRDE}
dY_t
=
F(Y_t)\,d{\bf Z}_t
+
F_0(Y_t)\,dt,
\end{equation}
where
\[
Y:[0,T]\to W, \qquad F:W\to\mathcal L(U,W),
\qquad
F_0:W\to W,
\]
and
\[
{\bf Z}=(Z,Z')
\in
\mathcal C_{\bf X}^\alpha([0,T],U)
\]
is controlled by
${\bf X}\in\mathcal D^\alpha([0,T],V)$.
We first relate this equation to a classical RDE driven by the canonical rough path lift of $Z$, and then use this relation to establish its pathwise well-posedness and the associated solution cocycle.

\subsection{Relation to classical rough differential equations} 
In this subsection, we relate the drift-containing CRDE~\eqref{main-diff-CRDE} to a classical RDE through the canonical rough path lift of the controlled driver $Z$. This relation will be used to establish pathwise well-posedness and continuity of the solution.
Let us first recall the concept of the controlled-against-controlled rough integral.

\begin{lemma}~\cite[Theorem~1]{Gubinelli04}
\label{thm:prelim-controlled-integral}
Let $\mathbf X=(X,\mathbb X) \in \mathcal{D}^\alpha([0,T],V)$ and
\[
{\bf Y}=(Y,Y')\in \mathcal{C}^\alpha_{\mathbf X}([0,T],\mathcal L(U,W)),
\qquad
{\bf Z}=(Z,Z')\in \mathcal{C}^\alpha_{\mathbf X}([0,T],U).
\]
Then the {\bf rough integral}
\begin{equation}\label{eq:prelim-controlled-integral}
\sint_s^t Y_r\,d{\bf Z}_r
:=
\lim_{|\mathcal P|\to 0}
\sum_{[u,v]\in\mathcal P}
\left(
Y_uZ_{u,v}
+
Y'_uZ'_u\mathbb X_{u,v}
\right) \in W 
\end{equation}
is well defined for every $s<t$, where $\mathcal P$ is an arbitrary partition of $[s,t]$. Moreover, there is a constant $C_\alpha$ such that
\begin{align*}
\left\| \sint_s^t Y_r\,d{\bf Z}_r-Y_sZ_{s,t}
-Y'_sZ'_s\mathbb X_{s,t}\right\|
\le
C_\alpha
\mathcal N_{\mathbf X}(Y,Z)
|t-s|^{3\alpha},
\end{align*}
where $\mathcal N_{\mathbf X}(Y,Z)$ is an increasing polynomial depending on $\|\mathbf X\|_{\alpha,[0,T]}$, 
$\|{\bf Y}\|_{\mathcal{C}_{\bf X}^\alpha([0,T])}$, and
$\|{\bf Z}\|_{\mathcal{C}_{\bf X}^\alpha([0,T])}$.
\end{lemma}

With the controlled-against-controlled rough integral at hand, we can now give a precise definition of a solution to the drift-containing CRDE~\eqref{main-diff-CRDE}.

\begin{definition}
We say that $\mathbf Y=(Y,Y')\in \mathcal{C}_{\bf X}^\alpha([0,T],W)$ is a {\bf solution of the drift-containing CRDE~\eqref{main-diff-CRDE}} with initial value $Y_0\in W$ if
\begin{equation}\label{main-integ-CRDE}
Y_t =Y_{0}+\sint_0^t F(Y_r)\,d{\bf Z}_r+\sint_0^t F_0(Y_r)\,dr,\qquad Y'_t=F(Y_t)Z'_t,\qquad \forall t\in[0,T].
\end{equation}
\mlabel{defn:sol}
\end{definition}

\begin{remark}
In Definition~\ref{defn:sol}, we regard the full controlled path
${\bf Y}=(Y,Y')$, rather than only its first component $Y$, as the
solution. Although the second component is not independent and is
determined by
$
Y'=F(Y)Z',
$
it records the controlled information entering the second-order
expansion. Indeed, Theorem~\ref{thm:solution-equivalence} below shows that
the CRDE and the corresponding RDE have the same first component $Y$,
but different Gubinelli derivatives,
\[
Y'_{\mathrm{CRDE}}=F(Y)Z',
\qquad
Y'_{\mathrm{RDE}}=F(Y).
\]
For this reason, in the subsequent study of random dynamics generated by CRDE, we take the states to be the pairs
${\bf Y}_t=(Y_t,Y'_t)$.
\end{remark}

To construct the second-order lift of the path $Z$, we consider the following controlled rough path.
Let $\mathbf X$ be a rough path and ${\bf Z}=(Z,Z')\in\mathcal{C}_{\bf X}^\alpha([0,T],U)$.
For $s\in [0,T],\,x\in U$ and $r\in [s,T]$, define
\begin{align*}
L_{x}:  U\to U\otimes U,\, y\mapsto x\otimes y\,\text{ and }\, 
L^s: [s,T]\to \mathcal L(U,U\otimes U),\, r \mapsto L^s_r:=L_{Z_{s,r}} .
\end{align*}

\begin{lemma}~\cite{LiGao2025} \label{lemma:lift-tensor-integrand-controlled}
The pair $(L^s,(L^s)')$ is an $\mathbf X$-controlled rough path with values in $\mathcal L(U,U\otimes U)$, where the Gubinelli derivative $(L^s)':[s,T]\to \mathcal L(V,\mathcal L(U,U\otimes U))$ is given by
\begin{align*}
(L^s)'_r :V\to \mathcal L(U,U\otimes U), \qquad v\mapsto L_{Z'_r v}.
\end{align*}
Moreover, its remainder term $R^{L^s}_{u,r}$ given in~\meqref{eq:remainder} is
$
R^{L^s}_{u,r} = L_{R^Z_{u,r}}.
$
\end{lemma}
 
For any $(s,t)\in \dt{0,T}$, define the canonical second level~\cite{LiGao2025} 
\begin{equation}
\label{eq:lift-canonical-second-level}
\mathbb Z: \dt{0,T}\to U\otimes U, \qquad (s,t)\mapsto \mathbb Z_{s,t} 
:=
\sint_s^t Z_{s,r} \otimes dZ_r 
:=
\sint_s^t L_r^s \,d{\bf Z}_r \in U\otimes U.
\end{equation}
Here the controlled-against-controlled rough integral is well defined by~\meqref{eq:prelim-controlled-integral}. 
Moreover, there exists an increasing function $M$ such that~\cite{LiGao2025}
\begin{equation}
\label{eq:lift-second-level-local-estimate}
\left\|
\mathbb Z_{s,t}
-
(Z'_s\otimes Z'_s)\mathbb X_{s,t}
\right\|
\le
M\bigl(\|\mathbf X\|_{\alpha,[0,T]},\|{\bf Z}\|_{\mathcal{C}_{\bf X}^\alpha([0,T])}\bigr)
|t-s|^{3\alpha}.
\end{equation}

\begin{lemma}~\cite{LiGao2025}~\label{thm:lift-canonical-lift-rough-path}
The canonical lift $\mathcal Z=(Z,\mathbb Z) \in \mathcal{D}^{\alpha}([0,T],U)$ is an $\alpha$-H\"older rough path.
Moreover,
\begin{equation*}
\|\mathcal Z\|_{\alpha,[0,T]}
\le
M_{\mathrm{lift}}
\bigl(
\|\mathbf X\|_{\alpha,[0,T]},
\|{\bf Z}\|_{\mathcal C_{\mathbf X}^\alpha([0,T])}
\bigr)
\end{equation*}
for an increasing function $M_{\mathrm{lift}}$.
\end{lemma}

We now consider the drift-containing RDE
\begin{equation}\label{eq:Z-RDE}
dY_{t}=F(Y_{t})d\mathcal Z_{t}+F_{0}(Y_{t})dt,
\end{equation}
which is obtained from~\meqref{eq-diff-RDE} by replacing ${\bf X}$ with $\mathcal Z$.

\begin{remark}
For every small interval $[s,t]\subset[0,T]$, the first component $Y$ of a solution $(Y,Y')$ to~\eqref{eq:Z-RDE} admits the local expansion 
\begin{equation}
\label{eq:solution-lifted-local-expansion}
Y_{s,t}
=
F(Y_s)Z_{s,t}
+
DF(Y_s)F(Y_s)\mathbb Z_{s,t}
+
F_0(Y_s)(t-s)
+
R^Y_{s,t},
\end{equation}
where
$
|R^Y_{s,t}|\lesssim |t-s|^{3\alpha}.
$
\end{remark}

To pave the way for further rigorous analysis of drift-containing CRDEs, 
we derive a correspondence between these equations and the drift-containing RDEs.

\begin{theorem}
\label{thm:solution-equivalence}
Assume $F\in C_b^2(W,\mathcal L(U,W))$ and $F_0\in C_b^1(W,W)$. Then $(Y, Y')\in\mathcal{C}_{\bf X}^\alpha([0,T],W)$ with $Y'=F(Y)Z'$ solves the drift-containing CRDE~\meqref{main-diff-CRDE}
if and only if $(Y,F(Y))\in\mathcal{C}_{\mathcal Z}^\alpha([0,T],W)$ solves the drift-containing RDE~\eqref{eq:Z-RDE}. 
\end{theorem}

\begin{proof} 
(Necessity) Suppose $(Y, Y'=F(Y)Z')$ solves the drift-containing CRDE~\meqref{main-diff-CRDE}. Then by $Y'_{t}=F(Y_{t})Z'_{t}$ and $F\in C_b^2(W,\mathcal L(U,W))$, we get
\[
F({\bf Y})=\big(F(Y),(F(Y))'=DF(Y)Y'\big)\in \mathcal{C}^\alpha_{\mathbf X}([0,T],\mathcal L(U,W)). 
\]
By Lemma~\ref{thm:prelim-controlled-integral},  
\begin{align}\label{eq:solution-controlled-integral-to-lifted}
\sint_{s}^{t}F(Y_{r})d{\bf Z}_{r}&\ =F(Y_{s})Z_{s,t}+F(Y)'_{s}Z'_{s}\mathbb{X}_{s,t}+O(|t-s|^{3\alpha})\nonumber\\
&\ =F(Y_{s})Z_{s,t}+DF(Y_{s})F(Y_{s})(Z'_{s}\otimes Z'_{s})\mathbb{X}_{s,t}+O(|t-s|^{3\alpha})\nonumber\\
&\ =F(Y_{s})Z_{s,t}+DF(Y_{s})F(Y_{s})\mathbb{Z}_{s,t}+O(|t-s|^{3\alpha}).\hspace{1cm}(\text{by }\eqref{eq:lift-second-level-local-estimate})
\end{align}
For the drift term, we have
\[
\sint_s^t F_0(Y_r)\,dr
=
F_0(Y_s)(t-s)
+
\sint_s^t
\bigl(
F_0(Y_r)-F_0(Y_s)
\bigr)\,dr.
\]
Since $F_0\in C_b^1(W,W)$,
\[
\|F_0(Y_r)-F_0(Y_s)\|
\le
\|DF_0\|\cdot \|Y_r-Y_s\|.
\]
As the controlled path $Y$ is $\alpha$-H\"older continuous,  it follows that $\|Y_r-Y_s\| \lesssim |r-s|^\alpha$ and 
\begin{align*}
\|\sint_{s}^{t}(F_{0}(Y_{r})-F_{0}(Y_{s}))dr\|\le &\ \sint_{s}^{t}\|F_{0}(Y_{r})-F_{0}(Y_{s})\|dr\\
\le &\ \|DF_{0}\| \sint_{s}^{t}\|Y_{r}-Y_{s}\|dr\\
\lesssim &\ \sint_{s}^{t}|r-s|^{\alpha}dr\\
\le &\ |t-s|^{\alpha}\sint_{s}^{t}dr\\
\le &\ |t-s|^{1+\alpha}\\
\le &\ |t-s|^{3\alpha}\hspace{1cm}(\text{by }\alpha\le \frac{1}{2}). 
\end{align*}
Consequently,
\begin{equation}
\label{eq:solution-drift-expansion}
\sint_s^t F_0(Y_r)\,dr
=
F_0(Y_s)(t-s)
+
O(|t-s|^{3\alpha}).
\end{equation}
Combining~\eqref{eq:solution-controlled-integral-to-lifted} and~\eqref{eq:solution-drift-expansion} yields that 
\[
Y_{s,t}
=
F(Y_s)Z_{s,t}
+
DF(Y_s)F(Y_s)\mathbb Z_{s,t}
+
F_0(Y_s)(t-s)
+
O(|t-s|^{3\alpha}).
\]
This is precisely the local expansion~\eqref{eq:solution-lifted-local-expansion}. Then by \[F\in C_b^2(W,\mathcal L(U,W)),\qquad F_0\in C_b^1(W,W),\] we obtain that 
$(Y,F(Y))\in \mathcal{C}_{\mathcal Z}^\alpha([0,T],W)$ is a solution of the drift-containing RDE~\eqref{eq:Z-RDE}.

(Sufficiency) Suppose that $(Y, F(Y))$ solves the drift-containing RDE~\eqref{eq:Z-RDE}.  Then 
\begin{align*}
Y_{s,t}&\ =F(Y_s)Z_{s,t}+DF(Y_s)F(Y_s)\mathbb Z_{s,t}+F_0(Y_s)(t-s)+R^Y_{s,t}\\
&\ =F(Y_{s})Z'_{s}X_{s,t}+F(Y_{s})R^Z_{s,t}+DF(Y_{s})F(Y_{s})\mathbb Z_{s,t}+F_{0}(Y_{s})(t-s)+R^{Y}_{s,t}\\
&\ =F(Y_{s})Z'_{s}X_{s,t}+O(|t-s|^{2\alpha}).
\end{align*}
Thus $(Y,Y'=F(Y)Z')$ is an $\mathbf X$-controlled rough path.
It remains to show that the integral equations coincide.
Let $\mathcal P$ be a partition of $[0,t]$.
The controlled-driver Riemann sum is
\[
S_{\mathcal P}^{\mathrm{ctrl}}
=
\sum_{[u,v]\in\mathcal P}
\left(
F(Y_u)Z_{u,v}
+
(F(Y))'_uZ'_u\mathbb X_{u,v}
\right),
\]
while the lifted rough Riemann sum is
\[
S_{\mathcal P}^{\mathrm{rp}}
=
\sum_{[u,v]\in\mathcal P}
\left(
F(Y_u)Z_{u,v}
+
DF(Y_u)F(Y_u)\mathbb Z_{u,v}
\right),
\]
with $\mathcal Z$-controlled rough path $(Y,F(Y))$. Since $Y'_u=F(Y_u)Z'_u$, we have 
\[
(F(Y))'_u
=
DF(Y_u)Y'_u
=
DF(Y_u)F(Y_u)Z'_u,
\]
and so
\[
(F(Y))'_uZ'_u\mathbb X_{u,v}
=
DF(Y_u)F(Y_u)(Z'_u\otimes Z'_u)\mathbb X_{u,v}.
\]
Therefore,  
\begin{align*}
\|S_{\mathcal P}^{\mathrm{rp}}-S_{\mathcal P}^{\mathrm{ctrl}}\|=& 
\bigg\|\sum_{[u,v]\in\mathcal P}\Big(DF(Y_u)F(Y_u)\mathbb{Z}_{u,v}-DF(Y_u)F(Y_u)(Z'_u\otimes Z'_u)\mathbb X_{u,v}\Big)\bigg\|\\
\le &\sum_{[u,v]\in\mathcal P}\|DF(Y_u)F(Y_u)(\mathbb Z_{u,v}-(Z'_u\otimes Z'_u)\mathbb X_{u,v})\|\\
\le &\sum_{[u,v]\in\mathcal P}\|DF(Y_u)\|\cdot \|F(Y_{u})\|\cdot \|\mathbb Z_{u,v}-(Z'_{u}\otimes Z'_{u})\mathbb{X}_{u,v}\|\\
\lesssim & \sum_{[u,v]\in\mathcal P}|v-u|^{3\alpha}\hspace{1cm}(\text{by~\eqref{eq:lift-second-level-local-estimate} and }F\in C_b^2(W,\mathcal L(U,W))).
\end{align*}
Since $3\alpha>1$, the last sum tends to $0$ as $|\mathcal P|\to0$.
Thus the lifted rough integral and the controlled-driver integral agree:
\[
\sint_0^t F(Y_r)\,d\mathbf Z_r
=
\sint_0^t F(Y_r)\,d\mathcal Z_r.
\]
The drift terms are identical in the two formulations. Thus $(Y, Y'=F(Y)Z')$ is a solution of the drift-containing CRDE~\meqref{main-diff-CRDE}.
This completes the proof.
\end{proof}

\begin{remark}
\label{rmk:controlled-formulation-significance}
Theorem~\ref{thm:solution-equivalence} shows that the CRDE is equivalent
to an RDE driven by the canonical lift
$\mathcal Z=(Z,\mathbb Z)$, but this does not make the controlled
formulation redundant.

\begin{enumerate}
\item
The natural data are the reference rough path ${\bf X}$ and the
controlled driver ${\bf Z}=(Z,Z')$. In contrast to the classical RDE
setting, the second level $\mathbb Z$ is not prescribed a priori.
It must be additionally reconstructed from the controlled data by
\[
\mathbb Z_{s,t}
=
\sint_s^t Z_{s,r}\otimes dZ_r.
\]
Thus, the controlled formulation provides the extra second-order
information needed to turn $Z$ into the rough path
$\mathcal Z=(Z,\mathbb Z)$.

\item
The local dynamical analysis also requires derivative, integrability,
and tempered estimates in terms of these controlled data; these are
developed in Section~\ref{sec:stationary-points}.
\end{enumerate}
Thus, the controlled formulation does not produce different state
trajectories from the lifted RDE, but preserves the additional
structure needed for stability and random dynamical analysis.
\end{remark}

The preceding correspondence simplifies in the special case where the controlled driver has identity Gubinelli derivative.

\begin{corollary}
Under the setting in Theorem~\ref{thm:solution-equivalence}, if $V= U$ and the controlled driver ${\bf Z}=(Z,Z')$ satisfies 
$$Z':[0,T]\to \mathcal L(V,U),\qquad t\mapsto \mathrm{Id}_{V},$$
then ${\bf Y}=(Y,Y'=F(Y))$ is a solution to~\meqref{main-diff-CRDE} if and only if ${\bf Y}$ is a solution to~\eqref{eq:Z-RDE}.
\end{corollary}

We now turn from the structural correspondence above to the pathwise well-posedness and stability of the drift-containing CRDE~\meqref{main-diff-CRDE}.

\begin{theorem} [Universal Limit Theorem] 
\label{thm:solution-pathwise-wellposedness}
Assume $F\in C_b^3(W,\mathcal L(U,W))$ and $F_0\in C_b^1(W,W)$. Let 
$$\mathbf X=(X,\mathbb X),\,\widetilde{\mathbf{X}}=(\widetilde{X},\widetilde{\mathbb X})\in\mathcal D^\alpha([0,T],V),$$ 
and
$${\bf Z}=(Z,Z')\in \mathcal{C}_{\bf X}^{\alpha}([0,T],U), \qquad \widetilde{\bf Z}=(\widetilde{Z},\widetilde{Z'})\in\mathcal C_{\widetilde{\bf X}}^\alpha([0,T],U).$$ 
Then
\begin{enumerate}
\item~\label{it-unique} 
For each $Y_0\in W$, there exists a unique solution $\mathbf Y \in\mathcal C_{\bf X}^\alpha([0,T],W)$ to the drift-containing CRDE~\eqref{main-diff-CRDE}.

\item~\label{it-contin} 
Let $\mathbf Y=(Y,Y')\in \mathcal C_{\bf X}^\alpha([0,T],W)$ and $\widetilde{\bf Y}=(\widetilde{Y},\widetilde{Y'})\in\mathcal C_{\widetilde{\bf X}}^\alpha([0,T],W)$ be the solutions to~\eqref{main-diff-CRDE} driven by ${\bf Z},\,\widetilde{\bf Z}$ with initial conditions $Y_0,\,\widetilde{Y}_0$ respectively. Then the following estimate holds true: 
\begin{align*}
    d_{{\bf X},\widetilde{\bf X};\alpha}(\mathbf Y,\widetilde{\bf Y})\le &\ C_{\alpha}M \bigg(T,\|F\|_{C_b^3},\|F_0\|_{C_b^1},\|Y'_0\|,\|\widetilde{Y}'_{0}\|,\|Z'_0\|,\|\widetilde{Z}'_0\|,\|{\bf Y}\|_{\mathcal C_{\bf X}^\alpha([0,T])},\|\widetilde{\bf Y}\|_{\mathcal C_{\widetilde{\bf X}}^\alpha([0,T])},\\
    &\ \|{\bf Z}\|_{\mathcal C_{\bf X}^\alpha([0,T])},\|\widetilde{\bf Z}\|_{\mathcal C_{\widetilde{\mathbf X}}^\alpha([0,T])},\|\mathbf X\|_{\alpha},\|\widetilde{\bf X}\|_{\alpha}\bigg)\times \bigg( d_{{\bf X},\widetilde{\bf X};\alpha}(\mathbf Z,\widetilde{\bf Z})+\|\mathbf X-\widetilde{\bf X}\|_\alpha\\
    &\ +\|Y_0-\widetilde{Y}_0\|+\|Y'_0-\widetilde{Y}'_0\|+\|Z_0-\widetilde{Z}_0\|+\|Z'_0-\widetilde{Z}'_0\| \bigg).
    \end{align*}
\end{enumerate}
\end{theorem}

\begin{proof}
\eqref{it-unique}. It follows from Theorem~\ref{thm:solution-equivalence} and the classical drift-containing rough path theory~\cite{FH20,FV10b}.

\eqref{it-contin}.
We follow the argument of~\cite[Theorem~4.1]{LiGao2025}.
Take 
\begin{align*}
\mathbf Y=&\ (Y_0+\sint_0^\cdot F(Y_{r})d{\bf Z}_{r}+\sint_0^\cdot F_0(Y_r)dr,F(Y)Z'), \\
\widetilde{\bf Y}=&\ (\widetilde{Y}_{0}+\sint_0^\cdot F(\widetilde{Y}_r)d\widetilde{\bf Z}_r+\sint_0^\cdot F_0(\widetilde{Y}_r)dr, F(\widetilde{Y})\widetilde{Z}').
\end{align*}
Carrying out direct estimate computations as in the cited argument, we complete the proof.
\end{proof}

\subsection{The solution cocycle of random drift-containing CRDEs}
\label{sec:solution-cocycle}
We now pass from the pathwise theory to the random setting. Assuming that the controlled driver is an ${\bf X}$-controlled rough path cocycle, we use its shift properties together with pathwise uniqueness to show that the time-$t$ solution maps generate a measurable solution cocycle.

Let $(\Omega,\mathcal F,\mathbb P,\{\theta_t\}_{t\in\mathbb R})$
be an invertible measure-preserving dynamical system.
Let ${\bf X}=(X,\mathbb X)$ be a rough path cocycle and
${\bf Z}=(Z,Z')$ an ${\bf X}$-controlled rough path cocycle
in the sense of Definition~\ref{def:prelim-rough-path-cocycle}.
Thus, for each $\omega\in\Omega$,
\[
{\bf X}(\omega)
=
\bigl(X(\omega),\mathbb X(\omega)\bigr)
\in\mathcal D^\alpha(\mathbb R,V),
\qquad
{\bf Z}(\omega)
=
\bigl(Z(\omega),Z'(\omega)\bigr)
\in\mathcal C_{{\bf X}(\omega)}^\alpha(\mathbb R,U).
\]
Assume that
\[
F\in C_b^3(W,\mathcal L(U,W)),
\qquad
F_0\in C_b^1(W,W).
\]
For each $\omega\in\Omega$, we consider the random version of~\meqref{main-integ-CRDE}
\begin{equation}
\label{eq:solution-map-equation}
Y_t(\omega)
=
Y_0(\omega)
+
\sint_0^t F(Y_r(\omega))\,d{\bf Z}_r(\omega)
+
\sint_0^t F_0(Y_r (\omega))\,dr,
\qquad t\ge0,
\end{equation}
with pathwise initial value $Y_0 (\omega)\in W$.
By Theorem~\ref{thm:solution-pathwise-wellposedness}, there is a unique controlled solution
\[
{\bf Y} (\omega)
=
\bigl(Y (\omega),Y'(\omega)\bigr)
\in
\mathcal C_{{\bf X}(\omega)}^\alpha(\mathbb R_{\ge0},W),
\]
with
$
Y'_t(\omega)
=
F(Y_t (\omega))Z'_t(\omega).
$
For fixed $\omega\in\Omega$, we define the pathwise solution map by
\[
\phi^\omega:
W\rightarrow
\mathcal C_{{\bf X}(\omega)}^\alpha(\mathbb R_{\ge0},W),
\qquad
Y_0(\omega)\mapsto
\phi^\omega(Y_0(\omega))
:=
{\bf Y} (\omega).
\]
We distinguish this map from the time-$t$ solution map
\[
\phi_t^\omega:
W \rightarrow W\times \mathcal L(V,W),
\qquad
Y_0(\omega)\mapsto
\phi_t^\omega(Y_0(\omega)):= {\bf Y}_t (\omega).
\]
For the pathwise initial condition $Y_0(\omega)\in W$ in~\eqref{eq:solution-map-equation}, the associated pathwise initial condition is 
$${\bf Y}_0(\omega):=(Y_0(\omega),F(Y_0(\omega))Z'_0(\omega))\in W\times \mathcal L(V,W)
$$ for the solution ${\bf Y}(\omega)\in\mathcal C_{\bf X(\omega)}^\alpha(\mathbb R,W)$. 
Thus, one has the process 
\begin{align}\label{random-sol-integ-CRDE}
\phi_\cdot^\cdot({\bf Y}_0):\mathbb R\times \Omega\to W\times \mathcal L(V,W),\qquad (t,\omega)\mapsto \phi_t^\omega({\bf Y}_0(\omega)):= \phi_t^\omega( Y_0(\omega)).
\end{align}

To establish the cocycle property of the time-$t$ solution maps, we first
record how the controlled rough integral and the drift integral behave
under a time shift.

\begin{lemma}\label{lemma-def-tilde-Y}
With the above setting and for any fixed $s\in\mathbb R$, define 
\begin{align*}
\widetilde{\bf Y}(\omega)=(\widetilde{Y}(\omega),\widetilde{Y}'(\omega)):\mathbb R\to W\times \mathcal L(V,W), \qquad r\mapsto \widetilde{\bf Y}_r(\omega):={\bf Y}_{s+r}(\omega).
\end{align*}  
Then  
\begin{enumerate}
\item~\label{it:cocycle-F} For $F\in C_{b}^{3}(W,\mathcal L(U,W)),$
\begin{equation*}
\sint_s^{s+t} F(Y_q(\omega))\,d{\bf Z}_q(\omega)
=
\sint_0^t F(\widetilde Y_r(\omega))\,d{\bf Z}_r(\theta_s\omega).
\end{equation*}
\mlabel{lem:solution-shift-controlled-integral}

\item~\label{it:cocycle-F0} For $F_0\in\mathcal L(W,W),$ 
\begin{equation*}
\sint_s^{s+t}F_0(Y_q(\omega))\,dq
=
\sint_0^tF_0(\widetilde Y_r(\omega))\,dr.
\end{equation*}
\end{enumerate}
\end{lemma}

\begin{proof}
By symmetry, we assume without loss of generality that $t\in \mathbb R_{\ge 0}$. 

\eqref{it:cocycle-F}. Let
\[ P: \; 0=r_0<r_1<\cdots<r_N=t
\]
be a partition of $[0,t]$.
The shifted partition of $[s,s+t]$ is
\[
s+r_0<s+r_1<\cdots<s+r_N.
\]
Then  
\begin{align*}
  \sint_s^{s+t} F(Y_q(\omega))\,d{\bf Z}_q(\omega)\overset{\eqref{eq:prelim-controlled-integral}}{=}& \lim_{|P|\to 0}\sum_{i=0}^{N-1} F(Y_{s+r_{i}}(\omega))Z_{s+r_{i},s+r_{i+1}}(\omega)+(F(Y(\omega)))'_{s+r_{i}}Z'_{s+r_{i}}(\omega)\mathbb{X}_{s+r_{i},s+r_{i+1}}(\omega)\\
  \overset{\eqref{shift-eq}}{=}& \lim_{|P|\to 0} \sum_{i=0}^{N-1} F(Y_{s+r_{i}}(\omega))Z_{r_{i},r_{i+1}}(\theta_{s}\omega)+F(Y_{s+r_{i}}(\omega))'Z'_{r_{i}}(\theta_{s}\omega)\mathbb{X}_{s+r_{i},s+r_{i+1}}(\omega)\\
  \overset{\eqref{eq:cocycle-relation}}{=}& \lim_{|P|\to 0}\sum_{i=0}^{N-1} F(\widetilde{Y}_{r_{i}}(\omega))Z_{r_{i},r_{i+1}}(\theta_{s}\omega)+F(\widetilde{Y}_{r_{i}}(\omega))'Z'_{r_{i}}(\theta_{s}\omega)\mathbb{X}_{r_{i},r_{i+1}}(\theta_{s}\omega)\\
  \overset{\eqref{eq:prelim-controlled-integral}}{=}& \sint_{0}^{t}F(\widetilde{Y}_{r}(\omega))d{\bf Z}_{r}(\theta_{s}\omega). 
\end{align*}

\eqref{it:cocycle-F0}. 
It follows from the ordinary change of variables $q=s+r$:
\[
\sint_s^{s+t}F_0(Y_q(\omega))\,dq
=
\sint_0^tF_0(Y_{s+r}(\omega))\,dr
=
\sint_0^tF_0(\widetilde Y_r(\omega))\,dr.
\]
This completes the proof.
\end{proof}

Next, we consider the cocycle property of the solution to~\eqref{eq:solution-map-equation} under the pathwise initial condition $Y_0(\omega)$.

\begin{theorem}
\label{thm:solution-cocycle}
Let $(\Omega,\mathcal F,\mathbb P,\{\theta_t\}_{t\in\mathbb R})$ be an invertible measure-preserving dynamical system, and let ${\bf Z}$ be an ${\bf X}$-controlled rough path cocycle in the sense of Definition~\ref{def:prelim-rough-path-cocycle}.
If $F\in C_b^3(W,\mathcal L(U,W))$ and $F_0\in C_b^1(W,W)$, then the map 
\begin{equation}
\phi:\mathbb R\times\Omega\times W\times \mathcal L(V,W)\to W\times \mathcal L(V,W) , \qquad \Big(t,\omega,{\bf Y}_{0}(\omega)\Big)\mapsto \phi_{t}^{\omega}({\bf Y}_{0}(\omega))
\mlabel{eq:solcocycle}
\end{equation}
given in \eqref{random-sol-integ-CRDE} is a measurable cocycle in the sense of Definition~\ref{def:prelim-measurable-cocycle}~\eqref{it-def:prelim-measurable-cocycle}. In this case, we refer to $\phi$ as the solution cocycle for~\eqref{eq:solution-map-equation}.
\end{theorem}

\begin{proof}
The measurability of $\phi_\cdot^\cdot(Y_0)$  follows from it being a random process.
It remains to prove the cocycle identity:
\begin{equation}
\label{eq:solution-cocycle}
\phi_{t+s}^\omega({\bf Y}_0(\omega))
=
\phi_t^{\theta_s\omega}
\bigl(
\phi_s^\omega({\bf Y}_0(\omega))
\bigr),
\qquad
\forall s, t\in \mathbb R.
\end{equation}
For any $r\in[0,T]$, set 
\[{\bf Y}_r(\omega):=\phi_r^\omega ({\bf Y}_0(\omega)),\qquad \widetilde{\bf Y}_r(\omega):={\bf Y}_{s+r}(\omega).\]
Then 
\begin{align*}
 \phi_{s+t}^{\omega}({\bf Y}_0(\omega))=&\ {\bf Y}_{s+t}(\omega)=\widetilde{\bf Y}_{t}(\omega),\\
  \phi_{t}^{\theta_{s}\omega}(\phi_{s}^{\omega}({\bf Y}_0(\omega)))=&\ \phi_{t}^{\theta_{s}\omega}({\bf Y}_{s}(\omega))
  =\phi_{t}^{\theta_{s}\omega}(\widetilde{\bf Y}_{0}(\omega)).
  \end{align*}
It follows from~\eqref{eq:solution-map-equation} that 
\[
Y_{s+t}(\omega)
=
Y_s(\omega)
+
\sint_s^{s+t}F(Y_q(\omega))\,d{\bf Z}_q(\omega)
+
\sint_s^{s+t}F_0(Y_q(\omega))\,dq.
\]
By Lemma~\ref{lemma-def-tilde-Y}, 
  \[\widetilde{Y}_{r}(\omega)=Y_{s}(\omega)+\sint_{0}^{r}F(\widetilde{Y}_{q}(\omega))d{\bf Z}_{q}(\theta_{s}\omega)+\sint_{0}^{r}F_{0}(\widetilde{Y}_{q}(\omega))dq,\]
namely, $\widetilde {\bf Y}(\omega)\in \mathcal C_{\bf X(\theta_s\omega)}^\alpha(\mathbb R,W)$ solves the \eqref{eq:solution-map-equation} in the shifted environment $\theta_s\omega$ with the associated pathwise initial condition
\[
\widetilde{\bf Y}_0(\omega)={\bf Y}_s(\omega)=\phi_s^\omega({\bf Y}_0(\omega)).
\]
Thus, $\phi_{t}^{\theta_{s}\omega}(\widetilde{\bf Y}_{0}(\omega))=\widetilde{\bf Y}_{t}(\omega)$, and so 
\[\phi_{t+s}^{\omega}({\bf Y}_0(\omega))=\widetilde{\bf Y}_{t}(\omega)=\phi_{t}^{\theta_{s}\omega}(\widetilde{\bf Y}_{0}(\omega))=\phi_{t}^{\theta_{s}\omega}(\phi_{s}^{\omega}
({\bf Y}_0(\omega))).\]
This completes the proof.
\end{proof}

To prepare for the study of the Lyapunov exponents, we now consider the differentiability of the solution to~\eqref{eq:solution-map-equation} with respect to its associated pathwise initial condition. 
To start with, we consider the $C^k$-regularity of the solution.

\begin{proposition}\label{thm:solution-Ck-regularity} 
For $k\in\mathbb{N}$, assume that $F\in C_{b}^{k+3}(W,\mathcal L(U,W))$ and $F_{0}\in C_{b}^{k+1}(W,W)$. 
Then, for every fixed $(t,\omega)$, the map $\phi_t^\omega$ given in~\eqref{random-sol-integ-CRDE} belongs to $C^k(W\times\mathcal L(V,W),W\times\mathcal L(V,W))$. Moreover, the Fr\'echet derivative $D_{y}\phi_t^\omega$ is the solution to the variational equation
\begin{equation}
\label{eq:solution-variational-equation-general}
d\xi_t
=
DF(\phi_t^{\omega}(y))\xi_t\,d{\bf Z}_t(\omega)
+
DF_0(\phi_t^{\omega}(y))\xi_t\,dt,
\qquad
\xi_0=\xi.
\end{equation}
\end{proposition}

\begin{proof}
By Theorem~\ref{thm:solution-equivalence} and~\cite{FH20,FV10b}, the map $\phi_t^\omega$ is in  $C^k(W\times \mathcal L(V,W),W\times \mathcal L(V,W))$. 
For $k\in\mathbb{N}$, differentiating the integral equation
\[
\phi_t^\omega(y)
=
y
+
\sint_0^tF(\phi_r^\omega(y))\,d{\bf Z}_r(\omega)
+
\sint_0^tF_0(\phi_r^\omega(y))\,dr
\]
with respect to $y\in W\times \mathcal L(V,W)$ in the direction $\xi$, we get
\[
D_y\phi_t^\omega[\xi]
=
\xi
+
\sint_0^tDF(\phi_r^\omega(y))D_y\phi_r^\omega[\xi]\,d{\bf Z}_r(\omega)
+
\sint_0^tDF_0(\phi_r^\omega(y))D_y\phi_r^\omega[\xi]\,dr,
\]
which is exactly~\eqref{eq:solution-variational-equation-general}.
\end{proof}

This regularity of the time-$t$ solution maps can be incorporated into the random dynamical systems framework as follows.

\begin{corollary} 
\label{cor:solution-Ck-cocycle}
Under the assumptions of Theorem~\ref{thm:solution-cocycle} and Proposition~\ref{thm:solution-Ck-regularity}, the map \[
\phi:\mathbb R\times\Omega\times W\times \mathcal L(V,W)\to W\times \mathcal L(V,W) , \qquad (t,\omega,{\bf Y}_{0}(\omega))\mapsto \phi_{t}^{\omega}({\bf Y}_{0}(\omega))
\]
is a measurable $C^k$-cocycle in the sense of Definition~\ref{def:prelim-measurable-cocycle}~\eqref{it-def:prelim-Ck-cocycle}.
\end{corollary}

\begin{proof}
The cocycle property and measurability follow from Theorem~\ref{thm:solution-cocycle}.
The $C^k$-regularity follows from Proposition~\ref{thm:solution-Ck-regularity}.
\end{proof}

\section{Stationary points and deterministic derivative estimates}
\label{sec:stationary-points}
In this section, we study the dynamics near a stationary point of the solution cocycle. We first introduce the corresponding linearized cocycle and establish the deterministic derivative estimates required for its integrability. We then derive nonlinear regularity estimates that will be used in the construction of local invariant manifolds. Throughout this section, we fix a time step $t_0\in\mathbb R_{>0}$.

\subsection{Stationary points}\label{subsec:sp-stationary}
We begin with the stationary point that will serve as the reference orbit for the subsequent linearization. Let
$(\Omega,\mathcal F,\mathbb P,\{\theta_t\}_{t\in\mathbb R})$
be an invertible measure-preserving dynamical system, and let $\phi$ be the solution cocycle constructed in Theorem~\ref{thm:solution-cocycle}. We first specify the stationary point of $\phi$ and then study the linearized cocycle along this stationary orbit.

Adopting the same pull-back attraction method as in~\cite{Arnold1998,CF1994}, we obtain a unique stationary point $Y:\Omega\to W\times \mathcal L(V,W)$ by
\begin{align}
Y_\omega=\big(Y(\omega),F(Y(\omega))Z'(\omega)\big)={\bf Y}(\omega)=\lim_{t\to \infty}\phi(t,\theta_{-t}\omega,{\bf Y}_0(\omega))=\lim_{t\to \infty}\phi_t^{\theta_{-t}\omega}({\bf Y}_0(\omega)),
\mlabel{eq:fixpt}
\end{align}
where ${\bf Y}_0(\omega)$ and $\phi$ are the associated pathwise initial condition and the solution cocycle of the drift-containing CRDE~\eqref{main-diff-CRDE}, respectively. 
Moreover, the stationary point does not depend on the choice of the associated initial condition ${\bf Y}_0$. 
For any fixed $\omega\in \Omega$ and $t\in\mathbb R_{\ge 0}$, suppose that $$\phi_t^\omega\in C^1\big(W\times \mathcal L(V,W),W\times \mathcal L(V,W)\big),$$ 
and define the map  
\begin{equation}
\label{eq:solution-linearized-cocycle}
\psi_t^\omega:W\times \mathcal L(V,W)\to W\times \mathcal L(V,W), \qquad x\mapsto \psi_t^\omega(x):=
D_{Y_\omega}\phi_t^\omega(x).
\end{equation}

\begin{proposition}\label{prop:solution-linearized-cocycle}
With the above setting, then  
\begin{equation*}
\psi_{t+s}^\omega
=
\psi_t^{\theta_s\omega}
\circ
\psi_s^\omega,
\qquad\forall
s,t\in \mathbb{R}_{\ge 0}.
\end{equation*}
Moreover, for $\xi\in W\times \mathcal L(V,W)$, the path $$\eta:=\psi_\cdot^\omega\xi:\mathbb R_{\ge 0}\to W\times \mathcal L(V,W),\qquad t\mapsto D_{ Y_\omega}\phi_t^\omega(\xi)$$ satisfies
\begin{equation*}
d\eta_t
=
DF(Y_{\theta_t\omega})\eta_t\,d{\bf Z}_t(\omega)
+
DF_0(Y_{\theta_t\omega})\eta_t\,dt.
\end{equation*}
\end{proposition}

\begin{proof}
First of all,
\begin{align*}
\psi_{t+s}^{\omega}=&\  D_{Y_{\omega}}\phi_{t+s}^{\omega}\hspace{1cm}(\text{by }~\eqref{eq:solution-linearized-cocycle})\\
=&\  D_{Y_{\omega}}(\phi_{t}^{\theta_{s}\omega}\circ \phi_{s}^{\omega})\hspace{1cm}(\text{by }~\eqref{eq:solution-cocycle})\\
=&\  D_{ Y_{\theta_{s}\omega}}\phi_{t}^{\theta_{s}\omega}\circ D_{Y_{\omega}}\phi_{s}^{\omega}\hspace{1cm}(\text{by the chain rule})\\
=&\  \psi_{t}^{\theta_{s}\omega}\circ \psi_{s}^{\omega}.\hspace{1cm}(\text{by }~\eqref{eq:solution-linearized-cocycle})
\end{align*}
Next, 
\begin{align*}
  \eta_{t}=&\ \psi_{t}^{\omega}(\xi)=D_{ Y_{\omega}}\phi_{t}^{\omega}(\xi)\hspace{1cm}(\text{by }~\eqref{eq:solution-linearized-cocycle})\\
  =&\  \xi+ \sint_{0}^{t} DF(\phi_{r}^{\omega}(Y_{\omega}))D_{ Y_{\omega}}\phi_{r}^{\omega}(\xi)d{\bf Z}_{r}(\omega)+\sint_{0}^{t}DF_{0}(\phi_{r}^{\omega}(Y_{\omega}))D_{Y_{\omega}}\phi_{r}^{\omega}(\xi)dr\hspace{1cm}(\text{by }~\eqref{eq:solution-variational-equation-general})\\
  =&\  \xi+\sint_{0}^{t}DF(Y_{\theta_{r}\omega})\eta_{r}d{\bf Z}_{r}(\omega)+\sint_{0}^{t}DF_{0}(Y_{\theta_{r}\omega})\eta_{r}dr.\hspace{1cm}(\text{by }\phi_{t}^{\omega}(Y_{\omega})=Y_{\theta_{t}\omega})
\end{align*}
Whence, \[d\eta_t
=
DF(Y_{\theta_t\omega})\eta_t\,d{\bf Z}_t(\omega)
+
DF_0(Y_{\theta_t\omega})\eta_t\,dt.\]
This completes the proof.
\end{proof}

The remainder of this subsection is devoted to preparing for the introduction of Lyapunov exponents and the proof of the existence of invariant manifolds. We begin by listing the quantitative inputs, including moment and smoothness conditions. 

\begin{assumption}\label{assump-F-integ}
We assume that the following conditions hold.
\begin{enumerate}
\item For $(F,F_0)$,  
\[F\in C_b^3(W,\mathcal L(U,W)),
\qquad
F_0\in C_b^1(W,W).\]
\item For the fixed time step $t_{0}\in \mathbb{R}_{>0}$,  
\[
\|\mathbf X(\omega)\|_{\alpha,[0,t_0]}
\in
\bigcap_{p\ge1}L^p(\Omega),\qquad 
\|{\bf Z(\omega)}\|_{\mathcal{C}_{\mathbf X(\omega)}^\alpha([0,t_0])}
\in
\bigcap_{p\ge1}L^p(\Omega).\]
\item The stationary point $Y$ given in~(\mref{eq:fixpt}) satisfies
\[ 
Y\in\bigcap_{p\ge1}L^p(\Omega;W\times \mathcal L(V,W)).
\]
\end{enumerate}
\end{assumption}

By the semiflow property of the solution map for the drift-containing CRDE~\eqref{main-diff-CRDE}, we have
\[(D_{z_{0}}\phi_{t}^{\omega})^{-1}=D_{\phi_{t}^{\omega}(z_{0})}\phi_{-t}^{\omega},\qquad \forall z_{0}\in W,\,t\in \mathbb{R}_{\ge0}.\]
Therefore, we can write
\[(\psi_{t}^{\omega})^{-1}:=D_{\phi_{t}^{\omega}(Y_{\omega})}\phi_{-t}^{\omega}.\]

\begin{remark}
Similar to~\cite[Remark~1.5]{GhaniVarzanehRiedel2025}, we have the following results.
\begin{enumerate}
\item If $\mathbf{Y}=(Y,Y')\in \mathcal{C}^\alpha_{\mathbf{X}}([0,T],W)$, then there exists a constant $C=C(\alpha,T)$ such that
\begin{align}
\|Y\|_{\alpha,[0,T]}
\le&\ 
C
\bigl(1+\|\mathbf X\|_{\alpha,[0,T]}\bigr)
\|{\bf Y}\|_{\mathcal{C}_{\bf X}^\alpha([0,T])}, \label{eq:prelim-Y-holder-bound}\\
\|{\bf Y}\|_{\mathcal{C}_{\bf X}^{\alpha}([a,c])}\le&\ (1+\|{\bf X}\|_{\alpha,[a,c]})\|{\bf Y}\|_{\mathcal{C}_{\bf X}^{\alpha}([a,b])}+\|{\bf Y}\|_{\mathcal{C}_{\bf X}^{\alpha}([b,c])}. \label{1.6}
\end{align}
\item For two ${\bf X}$-controlled rough paths ${\bf Y}=(Y,Y'),\,\widetilde{\bf Y}=(\widetilde{Y},\widetilde{Y}')\in \mathcal{C}_{{\bf X}}^{\alpha}([0,T],W)$, we have
\begin{align}\label{1.7}
\|(Y\widetilde{Y},(Y\widetilde{Y})')\|_{\mathcal{C}_{\bf X}^\alpha([0,T])}\lesssim (1+\|{\bf X}\|_{\alpha,[0,T]})\|{\bf Y}\|_{\mathcal{C}_{\bf X}^\alpha([0,T])}\|\widetilde{\bf Y}\|_{\mathcal{C}_{\bf X}^\alpha([0,T])},
\end{align}
\item Furthermore, for $G\in C_{b}^{2}(W,U)$, we get
\begin{align}\label{1.12}
\|(G(Y),G(Y)')\|_{\mathcal{C}_{\bf X}^{\alpha}([0,T])}\lesssim 1+\|{\bf Y}\|_{\mathcal{C}_{\bf X}^{\alpha}([0,T])}^{4}+\|{\bf X}\|_{\alpha,[0,T]}^{4}.
\end{align}
\end{enumerate}
\end{remark}

We next derive a deterministic bound for the derivative of the time-$t$ solution map, which will be used to establish the integrability of the linearized cocycle.

\begin{proposition} 
\label{prop:lin-deterministic-derivative-bound}
Assume that $F\in C_b^1(W,\mathcal L(U,W))$ and $F_{0}\in C_b^1(W,W)$. There exists an increasing polynomial-type function $Q_{\mathrm{lin}}$ such that for any $z_0\in W\times \mathcal L(V,W)$ and $\alpha\neq \frac{1}{2}$,
\begin{equation}
\label{eq:lin-deterministic-derivative-bound}
\sup_{0\le t\le T}
\log^+\|D_{z_{0}}\phi_{t}^\omega\|
\le
Q_{\mathrm{lin}}
\bigl(
\|z_{0}\|,\|\mathbf Z(\omega)\|_{\mathcal{C}_{\bf X(\omega)}^{\alpha}([0,T])},\|{\bf X(\omega)}\|_{\alpha,[0,T]}
\bigr),
\end{equation}
where $$\log^{+} x :=
\begin{cases}
	\log x, & x \ge 1, \\
	0, & 0 < x < 1.
\end{cases} $$   
\end{proposition}

\begin{proof}
By~\cite[Theorem~11.6]{FV10b} and our regularity assumption on $F_{0}$ and $F$, the time-$t$ solution map $\phi_{t}^\omega:W\times \mathcal L(V,W)\to W\times \mathcal L(V,W)$ of the drift-containing CRDE~\eqref{main-diff-CRDE} is differentiable with respect to the associated controlled initial conditions. 
Next we are going to prove estimate~\eqref{eq:lin-deterministic-derivative-bound}. The derivative satisfies the following equation:
\[dD_{z_{0}}\phi_{t}^{\omega}[\bar{z}]=D_{\phi_{t}^{\omega}(z_{0})}F(D_{z_0}\phi_{t}^{\omega}[\bar{z}])d{\bf Z}_{t}(\omega)+D_{\phi_{t}^{\omega}(z_{0})}
F_{0}(D_{z_{0}}\phi_{t}^{\omega}[\bar{z}])dt\]
with $D_{z_{0}}\phi_{0}^{\omega}[\bar{z}]=\bar{z}\in W\times \mathcal L(V,W)$. Moreover, 
\begin{align*}
 D_{z_0}\phi_t^\omega[\bar{z}]-D_{z_0}\phi_s^\omega[\bar{z}]
=\sint_{s}^{t}D_{\phi_{\tau}^{\omega}(z_0)}F(D_{z_0}\phi_\tau^\omega[\bar{z}])d{\bf Z}_{\tau}(\omega)+\sint_{s}^{t}D_{\phi_{\tau}^\omega(z_0)}F_{0}(D_{z_0}\phi_\tau^\omega[\bar{z}])d\tau
\end{align*}
By the definition of $F(Y)'$, 
\begin{align*}
&\ \|(D_{z_0}\phi^{\omega}(\bar{z}))'\|_{\mathcal{C}_{\bf X}^{\alpha}([s,t])}=\|D_{\phi_{t}^{\omega}(z_0)}F(D_{z_0}\phi_{t}^{\omega}[\bar{z}])\|_{\mathcal{C}_{\bf X}^{\alpha}([s,t])}\\
\lesssim &\ (1+\|{\bf X}\|_{\alpha,[s,t]})\|D_{\phi^{\omega}(z_{0})}F\|_{\mathcal{C}_{\bf X}^{\alpha}([s,t])}\|D_{z_0}\phi^{\omega}(\bar{z})\|_{\mathcal{C}_{\bf X}^{\alpha}([s,t])}\hspace{1cm}(\text{by~\eqref{1.7}})\\
\lesssim &\ (1+\|{\bf X}\|_{\alpha,[s,t]})(1+\|\phi^{\omega}(z_{0})\|_{\mathcal{C}_{\bf X}^{\alpha}([s,t])}^{4}+\|{\bf X}\|_{\alpha,[s,t]}^{4})\|D_{z_0}\phi^{\omega}(\bar{z})\|_{\mathcal{C}_{\bf X}^{\alpha}([s,t])}.\hspace{1cm}(\text{by~\eqref{1.12}})
\end{align*}
Since
\begin{align*}
&\ \|(\sint_{s}^{\cdot}Yd{\bf Z}_{\tau}, YZ')\|_{\mathcal{C}_{\bf X}^{\alpha}([s,t])}\\
=&\ \max\Big\{\|\sint_{s}^{\cdot}Yd{\bf Z}\|_{\infty},\|YZ'\|_{\infty},\|YZ'\|_{\alpha,[s,t]},\|\sint_{s}^{\cdot}YdZ-Y_{u}Z'_{u}X_{u,v}\|_{2\alpha,[s,t]}\Big\}\\
\lesssim &\ C_{\alpha,T}(|Y_{0}'|+\|{\bf Y}\|_{\mathcal{C}_{\bf X}^{\alpha}([s,t])})(|Z_{0}'|+\|{\bf Z}\|_{\mathcal{C}_{\bf X}^{\alpha}([s,t])})(1+\|{\bf X}\|_{\alpha,[s,t]}+\|{\bf X}\|_{\alpha,[s,t]}^{2})+\|{\bf Y}\|_{\mathcal{C}_{\bf X}^{\alpha}([s,t])}^{2}\hspace{1cm}(\text{by }\eqref{eq:prelim-Y-holder-bound})\\
\lesssim &\ (|Y_{0}'|+\|{\bf Y}\|_{\mathcal{C}_{\bf X}^{\alpha}([s,t])})(|Z_{0}'|+\|{\bf Z}\|_{\mathcal{C}_{\bf X}^{\alpha}([s,t])})(1+\|{\bf X}\|_{\alpha})+\|{\bf Y}\|^{2}_{\mathcal{C}_{\bf X}^{\alpha}([s,t])}+\|{\bf Y}\|_{\mathcal{C}_{\bf X}^{\alpha}([s,t])}\|{\bf Z}\|_{\mathcal{C}_{\bf X}^{\alpha}([s,t])}\|{\bf X}\|_{\alpha,[s,t]}^{2},
\end{align*}
we have
\begin{align*}
&\ \| \sint_{s}^{\cdot}D_{\phi_{\tau}^{\omega}(z_{0})}F(D_{z_{0}}\phi_{\tau}^{\omega}(\bar{z}))d{\bf Z}_{\tau}(\omega)\|_{\mathcal{C}_{\bf X}^{\alpha}([s,t])}\nonumber\\
=&\ \|\big(\sint_{s}^{\cdot}D_{\phi_{\tau}^{\omega}(z_{0})}F(D_{z_{0}}\phi_{\tau}^{\omega}(\bar{z}))d{\bf Z}_{\tau}(\omega),D_{\phi_{\tau}^{\omega}(z_{0})}
F(D_{z_0}\phi_{t}^{\omega}(\bar{z}))Z'\big)\|_{\mathcal{C}_{\bf X}^{\alpha}([s,t])}\nonumber\\
\lesssim &\ (1+\|{\bf X}\|_{\alpha}+\|{\bf X}\|_{\alpha}^{2})\|D_{\phi_{\tau}^{\omega}(z_{0})}F(D_{z_0}\phi^{\omega}(\bar{z}))\|_{\mathcal{C}_{\bf X}^{\alpha}([s,t])}\|{\bf Z}\|_{\mathcal{C}_{\bf X}^{\alpha}([s,t])}+\|D_{\phi_{\tau}^{\omega}(z_0)}F(D_{z_0}\phi^{\omega}(\bar{z}))\|_{\mathcal{C}_{\bf X}^{\alpha}([s,t])}^{2}\nonumber\\
\lesssim &\ (1+\|{\bf X}\|_{\alpha}+\|{\bf X}\|_{\alpha}^{2})(\|D_{z_{0}}\phi^{\omega}(\bar{z})\|+\|D_{z_0}\phi^{\omega}(\bar{z})\|^{2})\|{\bf Z}\|_{\mathcal{C}_{\bf X}^{\alpha}([s,t])}+Q(\|{\bf X}\|_{\alpha},\|\phi^{\omega}(z_{0})\|_{\mathcal{C}_{\bf X}^{\alpha}([s,t])}, \|{\bf Z}\|_{\mathcal{C}_{\bf X}^{\alpha}([s,t])})\nonumber\\
&\ \times \|D_{z_{0}}\phi^{\omega}(\bar{z})\|_{\mathcal{C}_{\bf X}^{\alpha}([s,t])}.\hspace{1cm}(\text{by $F$ being }C_{b}^{1})
\end{align*}
For the drift term, 
\begin{align*}
&\ \|(\sint_{s}^{\cdot}D_{\phi^{\omega}_{\tau}(z_{0})}F_{0}(D_{z_{0}}\phi^{\omega}_{\tau}(\bar{z}))d\tau,0)\|_{\mathcal{C}_{\bf X}^{\alpha}([s,t])}\nonumber\\
=&\ \max\Big\{\|\sint_{s}^{\cdot}D_{\phi^{\omega}_{\tau}(z_0)}F_{0}(D_{z_{0}}\phi^{\omega}_{\tau}(\bar{z}))d\tau\|_{\infty},\|\sint_{s}^{\cdot}D_{\phi^{\omega}_{\tau}
(z_{0})}F_{0}(D_{z_{0}}\phi^{\omega}_{\tau}(\bar{z}))d\tau\|_{2\alpha,[s,t]}\Big\}\nonumber\\
\lesssim &\ \sint_{s}^{t}\|D_{\phi^{\omega}_{\tau}(z_{0})}F_{0}(D_{z_{0}}\phi^{\omega}_{\tau}(\bar{z}))\|d\tau \cdot (t-s)^{-2\alpha}\nonumber\\
\le &\ \sup_{\tau\in [s,t]}\|D_{\phi^{\omega}_{\tau}(z_{0})}F_{0}\|\cdot \|D_{z_{0}}\phi^{\omega}_{\tau}(\bar{z})\|\cdot |t-s|^{1-2\alpha}\nonumber\\
\lesssim &\ |t-s|^{1-2\alpha}\|D_{z_{0}}\phi^{\omega}_{\cdot}(\bar{z})\|_{\mathcal{C}_{\bf X}^{\alpha}([s,t])},
\end{align*}
where there exist polynomials $Q_1,Q_2$ such that
\begin{align*}
&\ \|D_{z_{0}}\phi^{\omega}_{\cdot}(\bar{z})\|_{\mathcal{C}_{\bf X}^{\alpha}([s,t])}=\|\bar{z}+\sint_{0}^{\cdot}DFd{\bf Z}_{\tau}+\sint_{0}^{\cdot}DF_{0}d\tau\|_{\mathcal{C}_{\bf X}^{\alpha}([s,t])}\\
\le &\ \|\sint_{s}^{\cdot}D_{\phi^{\omega}_{\tau}(z_{0})}F(D_{z_{0}}\phi^{\omega}_{\tau}(\bar{z}))d{\bf Z}_{\tau}\|_{\mathcal{C}_{\bf X}^{\alpha}([s,t])}
+\|\sint_{s}^{\cdot}D_{\phi^{\omega}_{\tau}(z_{0})}F_{0}(D_{z_{0}}\phi^{\omega}_{\tau}(\bar{z}))d\tau\|_{\mathcal{C}^{\alpha}_{\bf X}([s,t])} +\|D_{z_{0}}\phi^{\omega}_{0}(\bar{z})\|_{\mathcal{C}^{\alpha}_{\bf X}([s,t])}\\
\lesssim &\ (1+\|{\bf X}\|_{\alpha}+\|{\bf X}\|_{\alpha}^{2})(\|D_{z_{0}}\phi^{\omega}(\bar{z})\|+\|D_{z_{0}}\phi^{\omega}(\bar{z})\|^{2})\|{\bf Z}\|_{\mathcal{C}_{\bf X}^{\alpha}([s,t])}+Q(\|{\bf X}\|_{\alpha},\|\phi^{\omega}(z_{0})\|_{\mathcal{C}_{\bf X}^{\alpha}([s,t])},\|{\bf Z}\|_{\mathcal{C}_{\bf X}^{\alpha}([s,t])})\\
&\ \times\|D_{z_{0}}\phi^{\omega}(\bar{z})\|_{\mathcal{C}_{\bf X}^{\alpha}([s,t])}+|t-s|^{1-2\alpha}Q_{1}\|D_{z_{0}}\phi^{\omega}(\bar{z})\|_{\mathcal{C}_{\bf X}^{\alpha}([s,t])}\\
\le &\ M\big[ (1+\|{\bf X}\|_{\alpha,[s,t]}^{2})(\|D_{z_{0}}\phi^{\omega}_{s}(\bar{z})\|+\|D_{z_{0}}\phi^{\omega}_{s}(\bar{z})\|^{2})\|{\bf Z}\|_{\mathcal{C}_{\bf X}^{\alpha}([s,t])}\\
&\ +|t-s|^{1-2\alpha}(Q_{2}(\|{\bf X}\|_{\alpha,[s,t]},\|\phi^{\omega}(z_{0})\|_{\mathcal{C}_{\bf X}^{\alpha}([s,t])},\|{\bf Z}\|_{\mathcal{C}_{\bf X}^{\alpha}([s,t])}))\|D_{z_{0}}\phi^{\omega}(\bar{z})\|_{\mathcal{C}_{\bf X}^{\alpha}([s,t])} \big].
\end{align*}
We choose a finite sequence $(\tau_{n})_{0\leq n\leq \widetilde{N}(\epsilon,{\bf X},{\bf Z},z_{0})}$ in $[0,T]$ such that 
$$\tau_{0}=0,\qquad \tau_{\widetilde{N}(\epsilon,{\bf X},{\bf Z},z_{0})}=T,\qquad \forall 0<\epsilon<1.$$
Assume that $$\tau_{n+1}-\tau_{n}=\tau,\qquad \forall 0\leq n\leq \widetilde{N}(\epsilon,{\bf X},{\bf Z},z_{0})-1,$$ and there exists $\epsilon\in (0,1)$ 
such that
\[\tau^{1-2\alpha}M\big[Q_{2}(\|{\bf X}\|_{\alpha,[0,T]},\|\phi^{\omega}(z_{0})\|_{\mathcal{C}_{\bf X}^{\alpha}([0,T])},\|{\bf Z}\|_{\mathcal{C}_{\bf X}^{\alpha}([0,T])})\big]=1-\epsilon.\]
Set  
$
M_\epsilon:=\frac{M}{\epsilon}.
$
Then, for every 
\[
0\le n\le \widetilde N(\epsilon,{\bf X},{\bf Z},z_0)-1,
\]
we have
\[\|D_{z_{0}}\phi^{\omega}(\bar{z})\|_{\mathcal C_{\bf X}^\alpha ([\tau_n,\tau_{n+1}])}\leq M_{\epsilon}(1+\|{\bf X}\|^{2}_{\alpha,[0,T]})\|{\bf Z}\|_{\mathcal{C}_{\bf X}^{\alpha}([0,T])}\|D_{z_{0}}\phi^{\omega}_{\tau_{n}}(\bar{z})\|.\]
Notice that
\begin{align*}
\widetilde{N}(\epsilon,{\bf X},{\bf Z},z_{0})=&\ \frac{T}{\tau}+1\\
=&\ T(M_{\epsilon}(Q_{2}(\|{\bf X}\|_{\alpha,[0,T]},\|\phi^{\omega}(z_{0})\|_{\mathcal{C}_{\bf X}^{\alpha}([0,T])},\|{\bf Z}\|_{\mathcal{C}_{\bf X}^{\alpha}([0,T])})))^{\frac{1}{1-2\alpha}}+1.
\end{align*} 
Then by~\cite{LiGao2025}, there exists a polynomial $Q_3$ such that
\[\|\phi^{\omega}(z_{0})\|_{\mathcal{C}_{\bf X}^{\alpha}([0,T])}\le Q_{3}(\|z_{0}\|,\|{\bf X}\|_{\alpha,[0,T]},\|{\bf Z}\|_{\mathcal{C}_{\bf X}^{\alpha}([0,T])}),\]
whence
\[\sup_{0\le n\le \widetilde{N}(\epsilon,{\bf X},{\bf Z},z_{0})-1}\|D_{z_{0}}\phi^{\omega}(\bar{z})\|_{\mathcal{C}_{\bf X}^{\alpha}([\tau_{n},\tau_{n+1}])}\le \|\bar{z}\| (M_{\epsilon}(1+\|{\bf X}\|_{\alpha,[0,T]}^{2})\|{\bf Z}\|_{\mathcal{C}_{\bf X}^{\alpha}([0,T])})^{\widetilde{N}(\epsilon,{\bf X},{\bf Z},z_{0})}.\]
From~\eqref{1.6}, for any $0\le n\le \widetilde{N}(\epsilon,{\bf X},{\bf Z},z_{0})-1$, we have
\[\|D_{z_{0}}\phi^{\omega}(\bar{z})\|_{\mathcal{C}_{\bf X}^{\alpha}([0,T])}\le (\widetilde{N}(\epsilon,{\bf X},{\bf Z},z_{0})+1)(\|{\bf X}\|_{\alpha,[0,T]}+1)(M_{\epsilon}(1+\|{\bf X}\|_{\alpha,[0,T]}^{2})\|{\bf Z}\|_{\mathcal{C}_{\bf X}^{\alpha}([0,T])})^{\widetilde{N}(\epsilon,{\bf X},{\bf Z},z_{0})}\|\bar{z}\|.\]
Thus, for any $\bar{z},z_{0}\in W\times \mathcal L(V,W)$ and $\alpha\neq \frac{1}{2}$, we get
\[\|D_{z_{0}}\phi^{\omega}(\bar{z})\|_{\mathcal{C}_{\bf X}^{\alpha}([0,T])}\leq \|\bar{z}\| \exp(Q_{3}(\|z_{0}\|,\|{\bf X}\|_{\alpha,[0,T]},\|{\bf Z}\|_{\mathcal{C}_{\bf X}^{\alpha}([0,T])})).\]
That is, there exists an increasing polynomial-type function $Q_{\mathrm{lin}}$ such that
\[\sup_{0\le t\le T}
\log^+\|D_{z_{0}}\phi_{t}^\omega\|\le \log^+\|D_{z_{0}}\phi^{\omega}_t\|_{\mathcal{C}_{\bf X}^{\alpha}([0,T])}\le Q_{\mathrm{lin}}
\bigl(
\|z_{0}\|,\|\mathbf Z(\omega)\|_{\mathcal{C}_{\bf X}^{\alpha}([0,T])},\|{\bf X(\omega)}\|_{\alpha,[0,T]}
\bigr).\]
This completes the proof.
\end{proof}

The preceding derivative estimate also provides quantitative control of the dependence of the solution on its associated initial condition.

\begin{corollary}
With the setting in Proposition~\ref{prop:lin-deterministic-derivative-bound}, there exists a polynomial $\widetilde{Q}_{\mathrm{lin}}$ such that for any $y,z\in W\times \mathcal L(V,W)$,
\begin{align}\label{phi-y-z-contin}
\sup_{0\le t\le T}\|\phi_t^\omega(y)-\phi_t^\omega(z)\|\le \|y-z\|\exp(\widetilde{Q}_{\mathrm{lin}}(\|y\|,\|z\|,\|\mathbf Z(\omega)\|_{\mathcal{C}_{\bf X}^{\alpha}([0,T])},\|{\bf X(\omega)}\|_{\alpha,[0,T]})).
\end{align}
\end{corollary}

\begin{proof}
It follows from Proposition~\ref{prop:lin-deterministic-derivative-bound} and
\[\phi_t^\omega(y)-\phi_t^\omega(z)=\sint_0^1D_{z+\theta(y-z)}\phi_t^\omega(y-z)d\theta.\qedhere\]
\end{proof}

For any fixed $z_0\in W\times \mathcal L(V,W)$, by the semiflow property, the maps 
$$\psi_t^{\omega}=D_{z_{0}}\phi_{t}^\omega: W\times \mathcal L(V,W)\to W\times \mathcal L(V,W),\qquad \forall t\in\mathbb{R}_{\geq 0},\,\omega\in \Omega$$ 
are invertible. We next estimate their inverses.

\begin{proposition}\label{prop:lin-deterministic-inverse-bound}
Under the assumptions of Proposition~\ref{prop:lin-deterministic-derivative-bound}, 
there exists an increasing polynomial type function $Q_{\mathrm{inv}}$ such that
\[\sup_{0\le t\le t_0}
\log^+\|(\psi_t^\omega)^{-1}\|
\le
Q_{\mathrm{inv}}
\bigl(\|z_0\|,
\|\mathbf Z(\omega)\|_{\alpha,[0,t_0]}
\bigr).\]
\end{proposition}

\begin{proof}
Fix $t_0 \in (0, T]$ and omit the sample parameter $\omega$ for brevity. Recall that for $t \in [0, t_0]$, 
the linearized flow $\psi_t := D_{z_0}\phi_t \in \mathcal L(W\times \mathcal L(V,W), W\times \mathcal L(V,W))$ satisfies 
\begin{equation}\label{eq:forward_psi}
\mathrm{d}\psi_t = DF(\phi_t(z_0))\psi_t \,\mathrm{d}\mathbf{Z}_t + DF_0(\phi_t(z_0))\psi_t \,\mathrm{d}t, \quad \psi_0 = \mathrm{Id}_{W\times \mathcal L(V,W)}.
\end{equation}
To establish the deterministic bound for $(\psi_t)^{-1}$ for all $t \in [0, t_0]$, we employ a time-reversal argument on $[0, t_0]$ in three steps.

\medskip
\noindent\textbf{Step 1.} 
Define the time-reversed rough path $\widetilde{\mathbf{X}} = (\widetilde{X}, \widetilde{\mathbb{X}})$ on $[0, t_0]$ by
\[
\widetilde{X}_{s,t} := X_{t_0-t} - X_{t_0-s} = - X_{t_0-t, t_0-s}, \qquad \widetilde{\mathbb{X}}_{s,t} := -\mathbb{X}_{t_0-t, t_0-s}, \qquad \forall (s,t)\in\dt{0,t_0}.
\]
Next, define the time-reversed driver $\widetilde{\mathbf{Z}} = (\widetilde{Z}, \widetilde{Z}')$ by
\[
\widetilde{Z}_t := Z_{t_0-t}, \qquad \widetilde{Z}'_t := Z'_{t_0-t}, \qquad \forall t \in [0, t_0].
\]
For any $0 \le s \le t \le t_0$, expanding the difference yields
\[
 \widetilde{Z}_{s,t} = Z_{t_0-t} - Z_{t_0-s} = -Z_{t_0-t, t_0-s} = -Z'_{t_0-t} X_{t_0-t, t_0-s} - R^Z_{t_0-t, t_0-s} = \widetilde{Z}'_t \widetilde{X}_{s,t} + R^{\widetilde{Z}}_{s,t},
\]
where $R^{\widetilde{Z}}_{s,t} := -R^Z_{t_0-t, t_0-s}$. Since $$\|R^{\widetilde{Z}}\|_{2\alpha,[0,t_0]} = \|R^Z\|_{2\alpha,[0,t_0]} < \infty,$$ we have $\widetilde{\mathbf{Z}}=(Z_{t_0-\cdot},Z'_{t_0-\cdot}) \in \mathcal{C}_{\widetilde{\mathbf{X}}}^\alpha([0, t_0],U)$ with Gubinelli derivative $\widetilde{Z}'$, and its controlled path norm satisfies
\[
\|\widetilde{\mathbf{Z}}\|_{\mathcal{C}_{\widetilde{\mathbf{X}}}^\alpha([0, t_0])} = \|\mathbf{Z}\|_{\mathcal{C}_{\mathbf{X}}^\alpha([0, t_0])}.
\]
For any controlled rough path $(Y, Y') \in \mathcal{C}_{\mathbf{X}}^\alpha([0, t_0])$, set 
$$(\widetilde{Y}_t, \widetilde{Y}'_t) := (Y_{t_0-t}, -Y'_{t_0-t}) \in \mathcal{C}_{\widetilde{\mathbf{X}}}^\alpha([0, t_0],W).$$ 
Then 
\begin{equation}\label{eq:time_reversal_integral}
\sint_a^b F(\widetilde{Y}_\tau) \,\mathrm{d}\widetilde{\mathbf{Z}}_\tau = -\sint_{t_0-b}^{t_0-a} F(Y_\tau) \,\mathrm{d}\mathbf{Z}_\tau, \qquad \forall 0 \le a \le b \le t_0.
\end{equation}

\medskip
\noindent\textbf{Step 2.} 
Fix any $t \in (0, t_0]$. Consider 
$$\widetilde{\psi}:[0, t]\to \mathcal L(W\times \mathcal L(V,W), W\times \mathcal L(V,W)),\qquad u \mapsto \widetilde{\psi}^u := \psi_{t-u} (\psi_t)^{-1}.$$
At the endpoints, we have \[\widetilde{\psi}^0 = \psi_t (\psi_t)^{-1} = \mathrm{Id}_{W\times \mathcal L(V,W)},\qquad \widetilde{\psi}^t = \psi_0 (\psi_t)^{-1} = (\psi_t)^{-1}.\]
For any $0 \le s < u \le t$, applying~\eqref{eq:forward_psi} on the interval $[t-u, t-s]$ gives
\[
\psi_{t-s} - \psi_{t-u} = \sint_{t-u}^{t-s} DF(\phi_\tau(z_0)) \psi_\tau \,\mathrm{d}\mathbf{Z}_\tau + \sint_{t-u}^{t-s} DF_0(\phi_\tau(z_0)) \psi_\tau \,\mathrm{d}\tau.
\]
Right-multiplying by $(\psi_t)^{-1}$ and using the time-reversal integral identity \eqref{eq:time_reversal_integral}, we obtain
\[
\widetilde{\psi}^s - \widetilde{\psi}^u = -\sint_{s}^{u} DF(\phi_{t-r}(z_0)) \widetilde{\psi}^r \,\mathrm{d}\widetilde{\mathbf{Z}}_r -\sint_{s}^{u}DF_{0}(\phi_{t-r}(z_{0}))\widetilde{\psi}^{r}dr.
\]
Rearranging terms demonstrates that $u \mapsto \widetilde{\psi}^u$ satisfies the following linear rough differential equation driven by $\widetilde{\mathbf{Z}}$:
\begin{equation}\label{eq:backward_crde}
\mathrm{d}\widetilde{\psi}^u = -DF(\phi_{t-u}(z_0)) \widetilde{\psi}^u \,\mathrm{d}\widetilde{\mathbf{Z}}_u - DF_0(\phi_{t-u}(z_0)) \widetilde{\psi}^u \,\mathrm{d}u, \qquad \widetilde{\psi}^0 = \mathrm{Id}_{W\times \mathcal L(V,W)}.
\end{equation}

\medskip
\noindent\textbf{Step 3.} 
Since \eqref{eq:backward_crde} is a linear rough differential equation driven by a controlled rough path with initial condition $\widetilde{\psi}^0 = \mathrm{Id}_{W\times \mathcal L(V,W)}$, we can apply the deterministic a priori controlled rough path estimate in Proposition~\ref{prop:lin-deterministic-derivative-bound} to \eqref{eq:backward_crde}. Specifically, the bound for $\widetilde{\psi}^u$ depends only on the regularity bounds of the vector fields $F, F_0$, the solution norm $\|\phi_{\cdot}(z_0)\|$, the driver norm $\|\widetilde{\mathbf{Z}}\|_{\mathcal{C}_{\widetilde{\mathbf{X}}}^\alpha([0, t_0])}$, and the initial condition $\|z_0\|$.

Consequently, there exists an increasing polynomial-type function $Q_{\text{inv}}$ such that for all $u \in [0, t]$,
\[
\|\widetilde{\psi}^u\| \le \exp\left( Q_{\text{inv}}\left( \|z_0\|, \|\mathbf{Z}\|_{\mathcal{C}_{\mathbf{X}}^\alpha([0, t_0])} \right) \right).
\]
Evaluating this estimate at $u = t$, where $\widetilde{\psi}^t = (\psi_t)^{-1} = (D_{z_0}\phi_t)^{-1}$, yields
\[
\left\| (\psi_t)^{-1} \right\| \le \exp\left( Q_{\text{inv}}\left( \|z_0\|, \|\mathbf{Z}\|_{\mathcal{C}_{\mathbf{X}}^\alpha([0, t_0])} \right) \right).
\]
Taking the supremum over $t \in [0, t_0]$ and using the definition of $\log^+(\cdot)$,
\[
\sup_{0 \le t \le t_0} \log^+ \left\| (\psi_t)^{-1} \right\| \le Q_{\text{inv}}\left( \|z_0\|, \|\mathbf{Z}\|_{\mathcal{C}_{\mathbf{X}}^\alpha([0, t_0])} \right),
\]
which completes the proof.
\end{proof}

\begin{corollary}\label{cor:lin-local-linear-control-integrable}
Under Assumption~\ref{assump-F-integ}, for fixed $t_0\in\mathbb{R}_{>0}$, 
\begin{equation*}
M_{\mathrm{lin}}(\omega)
:=
\sup_{0\le t\le t_0}
\log^+\|\psi_t^\omega\|
+
\sup_{0\le t\le t_0}
\log^+\|(\psi_t^\omega)^{-1}\|\in L^1(\Omega).
\end{equation*}
\end{corollary}

\begin{proof}
By Propositions~\ref{prop:lin-deterministic-derivative-bound} and~\ref{prop:lin-deterministic-inverse-bound},
\[
M_{\mathrm{lin}}(\omega)
\le
Q_{\mathrm{lin}}\bigl(\|\mathbf Z(\omega)\|_{\alpha;[0,t_0]}\bigr)
+
Q_{\mathrm{inv}}\bigl(\|\mathbf Z(\omega)\|_{\alpha;[0,t_0]}\bigr).
\]
Since
\[
\|\mathbf Z(\omega)\|_{\alpha;[0,t_0]}
\in
\bigcap_{p\ge1}L^p(\Omega),
\]
as well as $Q_{\mathrm{lin}}$ and $Q_{\mathrm{inv}}$ are polynomial-type functions, the right-hand side belongs to $L^1(\Omega)$.
Therefore $
M_{\mathrm{lin}}(\omega)\in L^1(\Omega)$.
\end{proof}

\subsection{Nonlinear regularity and deterministic derivative estimates}\label{subsec:manifold-nonlinear-regularity}
Having established deterministic bounds for the linearized cocycle, we now turn to the nonlinear dependence of the solution map on its associated initial condition. The key point is to control the variation of its Fr\'echet derivative locally and to obtain the tempered bounds required for the invariant manifold arguments.

\begin{proposition}\label{prop:manifold-derivative-lipschitz}
Let ${\bf \mathcal{Z}}$ be the canonical lifted rough path for the path $Z$ given in~\eqref{eq:lift-canonical-second-level}.  Then for every $t_0\in\mathbb{R}_{>0}$, there exists an increasing polynomial-type function $Q_{\mathrm{nl}}:\mathbb R_{\ge 0}^3\to\mathbb R_{\ge 0} $ such that, for any $y, z\in W\times \mathcal L(V,W)$ and $\rho\in(0,1]$ with $\|z\|\le\rho$,
\begin{equation}
\label{eq:manifold-derivative-lipschitz}
\sup_{0\le t\le t_0}
\left\|
D_{y+z}\phi_t^\omega
-
D_y\phi_t^\omega
\right\|
\le
f_{\rho,y}(\omega)\|z\|,
\end{equation}
where
\begin{equation*}
f_{\rho,y}(\omega)
:=
\exp
\left(
Q_{\mathrm{nl}}
\bigl(
\|y\|+\rho,
\|\mathcal Z(\omega)\|_{\alpha,[0,t_0]}, \|{\bf Z(\omega)}\|_{\mathcal C_{\bf X(\omega)}^\alpha([0,t_0])}
\bigr)
\right).
\end{equation*}
\end{proposition}

\begin{proof}
The proof follows from Proposition~\ref{prop:lin-deterministic-derivative-bound} by applying the argument used in the proof of~\cite[Proposition~1.21]{GhaniVarzanehRiedel2025}.
\end{proof}

Temperedness is a property that enables us to separate random variables with subexponential growth from those that grow too rapidly. We now recall its precise definition.

\begin{definition}~\cite{Arnold1998}\label{def:sp-tempered-radius}
Let $(\Omega,\mathcal F,\mathbb P,\{\theta_t\}_{t\in\mathbb R})$ be an invertible measure-preserving dynamical system. 
A random variable $\rho:\Omega\to[0,\infty)$ is called {\bf tempered} (resp.~$\theta_{t_0}$-tempered) if
\[\lim_{t\to \pm\infty}\frac{1}{|t|}\log^{+}\rho(\theta_t\omega)=0,\qquad (\text{resp.~}
\lim_{|n|\to\infty}
\frac1{|n|}
\log^+\rho(\theta_{nt_0}\omega)=0),\]
for almost every $\omega$.
\end{definition}

To verify the temperedness of the nonlinear coefficient $f_{\rho,y}$ in~\eqref{eq:manifold-derivative-lipschitz}, we fix the local radius $\rho=1$ and $y=Y_\omega\in W\times \mathcal L(V,W)$,  
\[f_{\mathrm{nl}}:=f_{1,Y_\omega}:\Omega\to \mathbb{R},\qquad \omega\mapsto \exp\bigg(Q_{\mathrm{nl}}\bigl(\|Y_\omega\|+1,\|\mathcal Z(\omega)\|_{\alpha,[0,t_0]},
\|{\bf Z(\omega)}\|_{\mathcal C_{\bf X(\omega)}^\alpha([0,t_0])}\bigr)\bigg).\]

Next, we consider the temperedness of $f_{\mathrm{nl}}$.

\begin{proposition}\label{lem:manifold-fnl-tempered}
In the setting of Proposition~\ref{prop:manifold-derivative-lipschitz}, assume that
\begin{enumerate}
\item the stationary point satisfies
\[
Y\in\bigcap_{p\ge 1}L^p(\Omega;W\times \mathcal L(V,W));
\]

\item the lifted driver and the controlled driver satisfy
\[
\|\mathcal Z(\omega)\|_{\alpha,[0,t_0]}
\in
\bigcap_{p\ge 1}L^p(\Omega),\qquad \|{\bf Z(\omega)}\|_{\mathcal C_{\bf X(\omega)}^\alpha([0,t_0])}\in
\bigcap_{p\ge 1}L^p(\Omega).
\]
\end{enumerate}
Then the random variable $f_{\mathrm{nl}}$ is $\theta_{t_0}$-tempered.
\end{proposition}

\begin{proof}
By Proposition~\ref{prop:manifold-derivative-lipschitz}, the function $Q_{\mathrm{nl}}$ is polynomial-type and increasing, and the variables
\[\|Y_\omega\|+1,\qquad
\|\mathcal Z(\omega)\|_{\alpha,[0,t_0]},\qquad \|{\bf Z(\omega)}\|_{\mathcal C_{\bf X(\omega)}^\alpha([0,t_0])}\]
have moments of all orders by assumption.
Therefore every polynomial expression in these variables belongs to $L^1(\Omega)$. Hence $\log^{+}f_{\mathrm{nl}}\in L^1(\Omega).$ 
So, for almost every $\omega$, we get
\[
\frac1{|n|}
\log^+
f_{\mathrm{nl}}(\theta_{nt_0}\omega)
\to0, \quad \text{ as } |n|\to \infty.
\]
Thus $f_{\mathrm{nl}}$ is $\theta_{t_0}$-tempered.
\end{proof}

We now apply the deterministic estimate near the stationary point.

\begin{coro}\label{prop:manifold-derivative-estimate-stationary}
For every $z\in W\times \mathcal L(V,W)$ with $\|z\|\le1$, one has
\begin{equation}
\label{eq:manifold-derivative-estimate-stationary}
\sup_{0\le t\le t_0}
\left\|
D_{Y_\omega+z}\phi_t^\omega
-
D_{Y_\omega}\phi_t^\omega
\right\|
\le
f_{\mathrm{nl}}(\omega)\|z\|.
\end{equation}
where 
\begin{equation*}
f_{\mathrm{nl}}(\omega)
=\exp \bigg(
Q_{\mathrm{nl}}
\bigl(
\|Y_\omega\|+1,
\|\mathcal Z(\omega)\|_{\alpha,[0,t_0]},
\|{\bf Z(\omega)}\|_{\mathcal C_{\bf X(\omega)}^\alpha([0,t_0])}
\bigr)\bigg).
\end{equation*}
\end{coro}

\begin{proof}
Apply Proposition~\ref{prop:manifold-derivative-lipschitz} with $y=Y_\omega,\,\rho=1.$ 
The coefficient $f_{\rho,y}(\omega)$ in~\eqref{eq:manifold-derivative-lipschitz} becomes exactly $f_{\mathrm{nl}}(\omega)$. Thus~\eqref{eq:manifold-derivative-estimate-stationary} follows.
\end{proof}

\section{Local invariant manifolds}
\label{sec:local-invariant-manifolds}
In this section, we construct local stable, unstable, and center manifolds around the stationary point of the solution cocycle. We first obtain the corresponding Oseledets splitting and establish the invariant manifolds for the time-$t_0$ cocycle, and then upgrade these discrete-time results to continuous time.
Throughout the section, we assume that the coefficients $(F,F_0)$ satisfy
\[F\in C_b^3(W,\mathcal L(U,W)),
\qquad
F_0\in C_b^1(W,W).\]

\subsection{Local stable, unstable and center manifolds}
We first establish the Oseledets splitting for the linearized cocycle and then use the resulting stable, unstable, and center subspaces to construct the corresponding local invariant manifolds. Throughout the remainder of the paper, we assume that $(\Omega,\mathcal F,\mathbb P,\{\theta_t\}_{t\in\mathbb R})$ is an ergodic invertible measure-preserving dynamical system.

\begin{lemma}~\label{lemma-29}
Under Assumption~\ref{assump-F-integ}, for $t\in\mathbb{R}_{\geq 0}$, let $$\psi_{\omega}^{t}:=D_{Y_{\omega}}\phi_{t}^{\omega}: W\times \mathcal L(V,W) \to W\times \mathcal L(V,W).$$ 
Then, on a full-measure set $\widetilde{\Omega}$ that is invariant under $\theta_t$, there 
exists a sequence of deterministic values~(i.e., Lyapunov exponents) $\mu_{k}<\cdots <\mu_{1}$, $\mu_{i}\in [-\infty ,\infty)$, 
and subspaces $H_{\omega}^{i}\subseteq W\times \mathcal L(V,W)$ such that
\begin{enumerate}
\item $W\times \mathcal L(V,W)=\oplus_{1\le i \le k}H_{\omega}^{i}$;
\item $\psi_{\omega}^{t}(H_{\omega}^{i})=H_{\theta_{t}\omega}^{i}$  for any $1\le i \le k$.
\item $\lim_{t\to\pm \infty}\frac{1}{t}\log\|\psi_{\omega}^t(\xi_{\omega})\|=\pm\mu_i $ if and only if $\xi_{\omega}\in H_\omega^i\setminus\{0\}$, where 
$$\psi_\omega^t(\xi_\omega):=(\psi_\omega^t)^{-1}(\xi_\omega),\qquad \forall t\in\mathbb R_{<0}.$$
\end{enumerate}
\end{lemma}

\begin{proof}
By Propositions~\ref{prop:lin-deterministic-derivative-bound} and~\ref{prop:lin-deterministic-inverse-bound}, for fixed $t_0\in\mathbb{R}_{>0}$, 
\begin{align*}
M_{\mathrm{lin}}(\omega)&\ 
:=
\sup_{0\le t\le t_0}
\log^+\|\psi_t^\omega\|
+
\sup_{0\le t\le t_0}
\log^+\|(\psi_t^\omega)^{-1}\|\in L^1(\Omega)\\
&\  \le
Q_{\mathrm{lin}}\bigl(\|\mathbf Z(\omega)\|_{\alpha;[0,t_0]}\bigr)
+
Q_{\mathrm{inv}}\bigl(\|\mathbf Z(\omega)\|_{\alpha;[0,t_0]}\bigr).
\end{align*}
Since
\[
\|\mathbf Z(\omega)\|_{\alpha;[0,t_0]}
\in
\bigcap_{p\ge 1}L^p(\Omega),
\]
as well as $Q_{\mathrm{lin}}$ and $Q_{\mathrm{inv}}$ are polynomial-type functions, the right-hand side belongs to $L^1(\Omega)$.
So, $M_{\mathrm{lin}}(\omega)\in L^1(\Omega)$.
Then by Proposition~\ref{prop:lin-deterministic-derivative-bound} again, 
$$\sup_{0\le t\le t_0}\log^{+}(\|\psi_{t_0-t}^{\theta_t\omega}\|)\in L^1(\Omega).$$
We thus complete the proof by applying~\cite[Proposition~2.9]{GhaniVarzanehRiedel2025} or~\cite[Theorem~1.21]{GVR2023}.
\end{proof}

For the linear Oseledets splitting of $W\times \mathcal L(V,W)$, we denote the linear subspaces of $W\times \mathcal L(V,W)$ by
\begin{align*}
S_{\omega}:=\oplus_{i:\mu_{i}\in\mathbb{R}_{<0}}H_{\omega}^{i},\quad U_{\omega}:=\oplus_{i:\mu_{i}\in\mathbb{R}_{>0}}H_{\omega}^{i},\quad C_{\omega}:=\oplus_{i:\mu_{i}=0}H_{\omega}^{i}.
\end{align*}

We first recall the notion of the local stable manifold.

\begin{definition}\label{def-stabManifold}~\cite[Theorem~2.10]{GVR2023}
Let $\phi$ be a Fr\'echet differentiable cocycle acting on the
measurable field of Banach spaces
$\{E_\omega\}_{\omega\in\Omega}$ with
$
	E_\omega=W\times\mathcal L(V,W),
$
and let $Y$ be a stationary solution of $\phi$. 
Then the {\bf local stable manifold} of $Y_\omega$ with rate $v$ is
defined by the set
\begin{align*}
		\mathcal W_{\loc}^{s,\nu}(\omega):=\Big\{Y_{\omega}+\Pi_{\omega}^0[\Gamma(v_{\omega})]:v_\omega\in S_\omega, \|v_{\omega}\|<\frac{1}{2h_1^\nu(\omega)}\min\{\frac{1}{2}h^{-1}(\frac{1}{2h_2^\nu(\omega)}),\inf_{n\in \mathbb N_0}\exp(nt_0\nu)\}\Big\},
	\end{align*}
	where $\Gamma(v_\omega)\in\Lambda_\omega^v$ denotes the unique solution to the fixed-point equation $I_\omega(v_\omega,\Gamma)=\Gamma$, and $\Pi_{\omega}^{j}:\prod_{i\in \mathbb N_0}E_{\theta^{i}\omega}\longrightarrow E_{\theta^{j}\omega}$ denotes the $j$-th coordinate projection.
Here, $h(x)=x^r g(x)$ in the sense of~\eqref{eq-GVR2023-assum2.5}, $I_\omega$ is defined in~\cite[Lemma~2.6]{GVR2023} and  
\begin{align*}
h_1^\nu(\omega):=&\ \sup_{n\in \mathbb N_0}  \exp(nt_0\nu)\|\psi^{\omega}_{nt_0}|_{S_{\omega}}\|  , \\
h_2^\nu(\omega):=&\ \sup_{n\in \mathbb N_0} \bigg( \exp(nt_0\nu) \sum_{0\le j\le n-1}\exp\Big(-j\nu(1+r)t_0\Big)\cdot f(\theta_{jt_0}\omega) \|\psi^{\theta_{(j+1)t_0}\omega}_{(n-j)t_0}|_{S_{\theta_{(j+1)t_0}\omega}} \| \cdot \|\Pi_{S_{\theta_{(j+1)t_0}\omega}\parallel \tilde{U}_{\theta_{(j+1)t_0}\omega}}\|\\
		&\ +\exp(nt_0\nu)\sum_{j\ge n}\exp\Big(-j\nu(1+r)t_0\Big)\cdot f(\theta_{jt_0}\omega) \|(\psi_{(j-n+1)t_0}^{\theta_{nt_0}\omega}|_{\tilde{U}_{\theta_{(j+1)t_0}\omega}})^{-1}\|\cdot \|\Pi_{\tilde{U}_{\theta_{(j+1)t_0}\omega}\parallel S_{\theta_{(j+1)t_0}\omega}}\|\bigg ),
	\end{align*}
	with the tempered measurable function $f:\Omega\to [0,\infty)$. Here $\tilde{U}_{\omega}:=U_{\omega}\oplus C_{\omega}$ and for $E=F\oplus H$,
\begin{align*}
    		\left\| \Pi_{F\parallel H} \right\|
    		:= \sup_{\substack{f\in F,\ h\in H\\ f+h\neq 0}}
    		\frac{\|f\|}{\|f+h\|} < \infty.
    	\end{align*}
\end{definition}

Now, we turn to the stable directions and obtain the corresponding local stable manifold.

\begin{theorem}[Local stable manifold]
\label{thm:manifold-local-stable}
Under the conditions of Lemma~\ref{lemma-29}, let $\mu_{i}\in\mathbb{R}_{<0}$ for some $1\le i \le k$. Fix an arbitrary time step $t_{0}\in\mathbb{R}_{>0}$ 
and a stable rate $\nu \in (0, -\mu^{-})$, where $\mu^{-}:=\max\{\mu_{i}: \mu_{i}\in\mathbb{R}_{<0}\}$. 
Then there exists a $\theta_{t_{0}}$-invariant subset $\widetilde{\Omega}\subseteq \Omega$ of 
full measure such that the following statements hold for any $\omega\in \widetilde{\Omega}$.
\begin{enumerate}
\item There exist two positive and finite random variables $\rho_{1}^{\nu}(\omega)$ and $\rho_{2}^{\nu}(\omega)$ such that
\[
\liminf_{p\to \infty}\frac 1p\log \rho_{i}^{\nu}(\theta_{pt_{0}}\omega)\geq 0,\qquad i=1,2,
\]
and 
\begin{align*}
&\ \Big\{ z\in W\times \mathcal L(V,W): \sup_{n\in\mathbb{N}_{0}}\exp(nt_{0}\nu)\|\phi_{nt_{0}}^{\omega}(z)-Y_{\theta_{nt_{0}}\omega}\|<\rho_{1}^{\nu}(\omega)\Big\}\subseteq \mathcal W_{\loc}^{s,\nu}(\omega)\\
&\ \subseteq \Big\{z\in W\times \mathcal L(V,W): \sup_{n\in\mathbb{N}_{0}}\exp(nt_{0}\nu)\|\phi_{nt_{0}}^{\omega}(z)-Y_{\theta_{nt_{0}}\omega}\|<\rho_{2}^{\nu}(\omega)\Big\}.
\end{align*}

\item The tangent space $T_{Y_{\omega}}\mathcal W_{\loc}^{s,\nu}(\omega)$ to the stable manifold $\mathcal W_{\loc}^{s,\nu}(\omega)$ at $Y_{\omega}$ is $S_{\omega}$, and for $n\geq N(\omega)$, 
\[\phi_{nt_{0}}^{\omega}(\mathcal W_{\loc}^{s,\nu}(\omega))\subseteq 
\mathcal W_{\loc}^{s,\nu}(\theta_{nt_{0}}\omega).\]

\item For $0< \nu_{1}\le \nu_{2}<-\mu^{-}$,  $$\mathcal W_{\loc}^{s,\nu_{2}}(\omega)\subseteq\mathcal W_{\loc}^{s,\nu_{1}}(\omega),$$ and for $n\geq N(\omega)$, 
\[\phi_{nt_{0}}^{\omega}(\mathcal W_{\loc}^{s,\nu_{1}}(\omega))\subseteq \mathcal W_{\loc}^{s,\nu_{2}}(\theta_{nt_{0}}\omega).\]

\item For any $z\in \mathcal W_{\loc}^{s,\nu}(\omega)$,
\[
\limsup_{n\to \infty }\frac 1n\log\|\phi_{nt_{0}}^{\omega}(z)-Y_{\theta_{nt_{0}}\omega}\|\le t_{0}\mu^{-} 
\]
and
\[
\limsup_{n\to \infty}\frac 1n \log \Big[ \sup\Big \{\frac{\|\phi_{nt_{0}}^{\omega}(\bar{z})-\phi_{nt_{0}}^{\omega}(z)\|}{\|\bar{z}-z\|},\, 
\bar{z}\neq z,\, \text{and } \bar{z},z\in \mathcal W_{\loc}^{s,\nu}(\omega) \Big \} \Big ]\le t_{0}\mu^{-}.
\]
\end{enumerate}
\end{theorem}

\begin{proof}
Define
\begin{equation*}
\Phi_\omega(z)
:=
\phi_{t_0}^\omega(Y_\omega+z)-Y_{\theta_{t_0}\omega}.
\end{equation*}
Then the linear part of $\Phi_\omega$ is
\begin{align*}
A_\omega
:=
D_0\Phi_\omega
=
D_{Y_\omega}\phi_{t_0}^\omega
=
\psi_{t_0}^\omega,
\end{align*}
and the nonlinear remainder is
\begin{align*}
P_\omega(z)
:=
\Phi_\omega(z)-A_\omega z.
\end{align*}
Equivalently, we write
\begin{equation*}
P_\omega(z)
=
\phi_{t_0}^\omega(Y_\omega+z)
-
\phi_{t_0}^\omega(Y_\omega)
-
\psi_{t_0}^\omega z.
\end{equation*}
By~\cite[Theorem~2.10]{GVR2023}, it suffices to verify that
\begin{enumerate}
\item~\label{it:proof-stabmanifold-1} the linearized cocycle $\psi$ around the stationary point $Y$ satisfies the integrability condition
\begin{align}~\label{rela-GVR2023-1.1}
\log^{+}\|\psi_{\omega}\|\in L^{1}(\Omega);
\end{align}
\item~\label{it:proof-stabmanifold-2} for any $z_1,\, z_2\in  W\times \mathcal L(V,W)$ with $\|z_1\|,\,\|z_2\|<1$, 
\begin{align}~\label{eq-GVR2023-assum2.5}
\|P_{\omega}(z_2)-P_\omega(z_1)\|\le \|z_2-z_1\| f(\omega) h(\|z_1\|+\|z_2\|)
\end{align}
where the random variable $f:\Omega\to \mathbb R$ is $\theta_{t_0}$-tempered and $h(x)=x^rg(x)$ for some $r>0$ where $g:\mathbb R\to \mathbb R_{\ge 0}$ is an increasing $C^1$ function.
\end{enumerate}

First of all, the~\eqref{it:proof-stabmanifold-1} follows from~Proposition~\ref{prop:lin-deterministic-derivative-bound}. Next, we are going to prove the~\eqref{it:proof-stabmanifold-2}. Here we take $h(x):=x(x+1)$. This choice is suitable, since $h$ can be written as $h(x)=x^r g(x)$ with $r=1\ge 0$, where $g(x)=x+1$ is an increasing $C^1$-function. 
Moreover, we set $f:=f_{\mathrm{nl}} : \Omega\to \mathbb{R}$. By Proposition~\ref{lem:manifold-fnl-tempered}, this $f$ satisfies the condition in~\eqref{it:proof-stabmanifold-2}. 
Now, it remains to verify that 
\begin{equation*}
\|P_\omega(z_2)-P_\omega(z_1)\|
\le
f_{\mathrm{nl}}(\omega)
(\|z_1\|+\|z_2\|)
\|z_2-z_1\|.
\end{equation*}
By the definition of $P_{\omega}$, 
\begin{align*}
P_{\omega}(z_{2})-P_{\omega}(z_{1})=& \ \phi_{t_{0}}^{\omega}(Y_{\omega}+z_{2})-\psi_{t_{0}}^{\omega}(z_{2})-\phi_{t_{0}}^{\omega}(Y_{\omega}+z_{1})+\psi_{t_{0}}^{\omega}(z_{1})\\
=&\ \phi_{t_{0}}^{\omega}(Y_{\omega}+z_{2})-\phi_{t_{0}}^{\omega}(Y_{\omega}+z_{1})-\psi_{t_{0}}^{\omega}(z_{2}-z_{1})\\
=& \ \phi_{t_{0}}^{\omega}(Y_{\omega}+z_{2})-\phi_{t_{0}}^{\omega}(Y_{\omega}+z_{1})-D_{Y_{\omega}}\phi_{t_{0}}^{\omega}(z_{2}-z_{1})\\
=& \ \sint_{0}^{1}D_{Y_{\omega}+z_{\theta}}\phi_{t_{0}}^{\omega}(z_{2}-z_{1})d\theta-D_{Y_{\omega}}\phi_{t_{0}}^{\omega}(z_{2}-z_{1})\hspace{1cm}(\text{by the mean value formula})\\
=& \ \sint_{0}^{1}(D_{Y_{\omega}+z_{\theta}}\phi_{t_{0}}^{\omega}-D_{Y_{\omega}}\phi_{t_{0}}^{\omega})(z_{2}-z_{1})d\theta\\
\le & \sint_{0}^{1}f_{\mathrm{nl}}(\omega)\|z_{\theta}\| d\theta \|z_{2}-z_{1}\| \hspace{1cm}(\text{by \eqref{eq:manifold-derivative-estimate-stationary}})\\
\leq & \sint_{0}^{1}f_{\mathrm{nl}} (\omega)d\theta \|z_{2}-z_{1}\|(\|z_{1}\|+\|z_{2}\|)\\
& \ \hspace{1cm}(\text{by }\|z_{\theta}\|\leq (1-\theta)\|z_{1}\|+\theta\|z_{2}\|\leq \|z_{1}\|+\|z_{2}\|\leq 1)\\
=&\  f_{\mathrm{nl}}(\omega)(\|z_{1}\|+\|z_{2}\|)\|z_{2}-z_{1}\|.
\end{align*}
This completes the proof.
\end{proof}

The unstable manifold is formulated in terms of backward orbits. 

\begin{definition}~\cite[Theorem~2.17]{GVR2023}
With the same setting as in Definition~\ref{def-stabManifold}, the {\bf local unstable manifold} of $Y_\omega$ with rate $\nu$ is the set
	\begin{align*}
		\mathcal W_{\loc}^{u,\nu}(\omega):=\Big\{Y_{\omega}+\Pi_{\omega}^0[\Gamma(u_{\omega})]: u_{\omega}\in U_\omega, \|u_{\omega}\|<\frac{1}{2\tilde{h}_1^\nu(\omega)}\min\{\frac{1}{2}h^{-1}(\frac{1}{2\tilde{h}_2^\nu(\omega)}),\inf_{n\in \mathbb N_0}\exp(nt_0 \nu)\}\Big\},
	\end{align*}
where 
\begin{align*}
\tilde{h}_{1}^\nu(\omega):=&\ \sup_{n\in \mathbb N_0}\exp(nt_0\nu)\|(\psi_{nt_0}^{\theta_{nt_0}\omega}|_{ U_{\omega}})^{-1}\|,\\
\tilde{h}_2^\nu(\omega):=&\ \sup_{n\in \mathbb N_0} \bigg(\exp(nt_0\nu)\sum_{k\ge n}\exp\Big(-\nu(k+1)(1+r)t_0\Big) \cdot f(\theta_{(k+1)t_0}\omega) \cdot \|\psi_{(k-n)t_0}^{\theta_{kt_0} \omega}|_{\tilde{S}_{\theta_{kt_0}\omega}}\|\cdot \|\Pi_{\tilde{S}_{\theta_{kt_0} \omega}\parallel U_{\theta_{kt_0} \omega}}\|\\
&\ +\exp(nt_0\nu)\sum_{0\le k\le n-1}\exp\Big(-\nu(n-k)(1+r)t_0\Big) \cdot f(\theta_{(n-k)t_0}\omega) \cdot \|(\psi_{(k+1)t_0}^{\theta_{n}\omega}|_{ U_{\theta_{(n-1-k)t_0}\omega}})^{-1}\|\\
&\ \cdot \|\Pi_{ U_{\theta_{(n-1-k)t_0}\omega}\parallel\tilde{S}_{\theta_{(n-1-k)t_0}\omega}}\| \bigg).
	\end{align*}
Here $\tilde{S}_{\omega}:=S_\omega \oplus C_\omega$.
\end{definition}

We next consider the unstable directions and construct the corresponding local unstable manifold.

\begin{theorem}[Local unstable manifold]
\label{thm:manifold-local-unstable}
Under the conditions of Lemma~\ref{lemma-29}, let $\mu_{1}\in\mathbb{R}_{>0}$ and $0<\nu <\mu^{+}$, where $\mu^+:=\min\{\mu_i:\mu_i\in\mathbb{R}_{>0}\}$. 
Fix an arbitrary time step $t_{0}\in\mathbb{R}_{>0}$. Then there exist a $\theta_{t_0}$-invariant subset $\widetilde{\Omega}\subseteq \Omega$ of full measure 
and a family of immersed submanifolds $\mathcal W_{\mathrm{loc}}^{u,\nu}(\omega)$ of $W\times \mathcal L(V,W)$ such that the following statement is valid for any $\omega\in \widetilde{\Omega}$. 
\begin{enumerate}
\item There exist two positive and finite random variables $\tilde{\rho}_{1}^{\nu}(\omega)$ and $\tilde{\rho}_{2}^{\nu}(\omega)$ such that
\[
\liminf_{p\to -\infty}\frac 1p\log \tilde{\rho}_{i}^{\nu}(\theta_{pt_{0}}\omega)\geq 0,\; i=1,2,
\]
and
\renewcommand{\arraystretch}{1.0}
\[
\begin{aligned}
	&\Bigl\{ z_{\omega} \in W \times \mathcal L(V,W) :\ 
	\exists \{z_{\theta_{-nt_0}\omega}\}_{n\in\mathbb{N}_{\ge 1}}
	\text{ such that } \\
	&\qquad \phi_{mt_0}^{\theta_{-nt_0}\omega}(z_{\theta_{-nt_0}\omega})
	= z_{\theta_{(m-n)t_0}\omega},\quad
	\text{for all }0\le m\le n,\text{ and}\\
	&\qquad \sup_{n\in\mathbb{N}_{\ge 0}}\exp(nt_0\nu)
	\|z_{\theta_{-nt_0}\omega}-Y_{\theta_{-nt_0}\omega}\|
	< \tilde{\rho}_1^\nu(\omega) \Bigr\}
	\subseteq \mathcal W_{\mathrm{loc}}^{u,\nu}(\omega) \\
	&\subseteq \Bigl\{ z_{\omega} \in W \times \mathcal L(V,W) :\ 
	\exists \{z_{\theta_{-nt_0}\omega}\}_{n\ge 1}
	\text{ such that } \\
	&\qquad \phi_{mt_0}^{\theta_{-nt_0}\omega}(z_{\theta_{-nt_0}\omega})
	= z_{\theta_{(m-n)t_0}\omega},\quad
	\text{for all }0\le m\le n,\text{ and}\\
	&\qquad \sup_{n\in\mathbb{N}_{\ge 0}}\exp(nt_0\nu)
	\|z_{\theta_{-nt_0}\omega}-Y_{\theta_{-nt_0}\omega}\|
	< \tilde{\rho}_2^\nu(\omega) \Bigr\}.
\end{aligned}
\]

\item $T_{Y_{\omega}}\mathcal W_{\mathrm{loc}}^{u,\nu}(\omega)=U_{\omega}$ and for any $n\ge N(\omega)$,
\[\mathcal W_{\mathrm{loc}}^{u,\nu}(\omega)\subseteq \phi_{nt_{0}}^{\theta_{-nt_{0}}\omega}(\mathcal W_{\mathrm{loc}}^{u,\nu}(\theta_{-nt_{0}}\omega)).\]

\item For $0<\nu_{1}\le \nu_{2}<\mu^{+}$,  
\[\mathcal W_{\mathrm{loc}}^{u,\nu_{2}}(\omega)\subseteq \mathcal W_{\mathrm{loc}}^{u,\nu_{1}}(\omega),\] 

\item For $n\ge N(\omega)$,  
    \[\mathcal W_{\mathrm{loc}}^{u,\nu_{1}}(\omega)\subseteq \phi_{nt_{0}}^{\theta_{-nt_{0}}\omega}(\mathcal W_{\mathrm{loc}}^{u,\nu_{2}}(\theta_{-nt_{0}}\omega)).\]
    
\item For any $z_{\omega}\in \mathcal W_{\mathrm{loc}}^{u,\nu}(\omega)$,  
\[\limsup_{n\to \infty}\frac 1n\log\|z_{\theta_{-nt_{0}}\omega}-Y_{\theta_{-nt_{0}}\omega}\|\le -t_{0}\mu^{+},\]
\[\limsup_{n\to -\infty} \frac 1n\log\left[\sup\left\{\frac{\|\bar{z}_{\theta_{-nt_{0}}\omega}-z_{\theta_{-nt_{0}}\omega}\|}{\|\bar{z}_{\omega}-z_{\omega}\|}, \bar{z}_{\omega}\neq z_{\omega}, \text{ and }\bar{z}_{\omega}, z_{\omega}\in \mathcal W_{\mathrm{loc}}^{u,\nu}(\omega)\right\}\right]\le -t_{0}\mu^{+}.\]
\end{enumerate}
\end{theorem}

\begin{proof}
By~\cite[Theorem~2.17]{GVR2023}, it suffices to verify that~\eqref{rela-GVR2023-1.1} and~\eqref{eq-GVR2023-assum2.5} hold. 
These conditions are verified as in the proof of Theorem~\ref{thm:manifold-local-stable}, which completes the proof.
\end{proof}

We finally turn to the center directions and recall the notion of a local center manifold.

\begin{definition}\cite[Definition~3.11]{center2025}
Let $Y:\Omega\to W\times \mathcal L(V,W)$ be the stationary point of the solution cocycle $\phi$ in~(\mref{eq:solcocycle}). The {\bf local center manifold} of $Y_\omega$ with rate $\nu$ is the set
\begin{align*}
\mathcal W_{\mathrm{loc}}^{c,\nu}(\omega):=&\ \Big\{Y_{\omega}+\xi_{\omega}\in W\times \mathcal L(V,W): \exists \Gamma\in \Sigma_{\omega}^\nu \text{ with } \Pi _{\omega}^0[\Gamma]=\xi_{\omega}\text{ and }\phi^{\theta_{nt_0}\omega}_{mt_0}(\Pi_{\omega}^n[\Gamma]+Y_{\theta_{nt_0}\omega})\\
&\ =\Pi_{\omega}^{n+m}[\Gamma]+Y_{\theta_{(n+m)t_0}\omega},\; \forall m,n \in \mathbb N_0 \Big\},
\end{align*}
where 
\[\Sigma_{\omega}^\nu:= \Big\{\Gamma\in  W\times \mathcal L(V,W):\|\Gamma\|=\sup_{j\in\mathbb N_0}\|\Pi_{\omega}^j[\Gamma]\|\exp(-\nu|j|t_0)<\infty \Big\}.\]
\end{definition}

Regarding the existence of the center manifold, we have the following result.

\begin{theorem}[Local center manifold]
\label{thm:manifold-local-center}
Under the conditions of Lemma~\ref{lemma-29}, assume that for some $1\le i_{c} \le k$, we have $\mu_{i_{c}}=0$. Fix any time step $t_{0}\in\mathbb{R}_{>0}$ and $0<\nu <\min\{\mu_{i_{c}-1},-\mu_{i_{c}+1}\}$ with $\mu_{0}=\infty$ if $i_{c}=1$. Let $\mathbb{N}t_{0}=\{nt_{0}:n\in \mathbb N_0 \}$. Then there exist a $\theta_{t_{0}}$-invariant subset $\widetilde{\Omega}\subseteq \Omega$ of full measure, a continuous cocycle $\bar{\phi}:\mathbb N t_{0}\times \widetilde{\Omega}\times W\times \mathcal L(V,W)\to W\times \mathcal L(V,W)$ and a positive random variable $\rho^{c}:\widetilde{\Omega}\to (0,\infty)$ such that the following statements is true. 
\begin{enumerate}
	\item $$\liminf_{n\to \infty}\frac 1n \log \rho^{c}(\theta_{nt_{0}}\omega)\geq 0.$$
	\item For any $\omega\in \widetilde{\Omega}$, there exists a function $$h_{\omega}^{c}: C_{\omega}\to \mathcal W_{\mathrm{loc}}^{c,\nu}(\omega)\subseteq W\times \mathcal L(V,W),\qquad v_{\omega}\mapsto h_\omega^c(v_\omega):=\Pi_\omega^0(\Gamma),$$
	where $\Gamma$ is the unique fixed point that satisfies $I_\omega(v_\omega,\Gamma)=\Gamma$ and the map $I_\omega:C_\omega\times \Sigma_\omega^\nu\cap B(0,R(\omega))\to \Sigma_\omega^\nu$ is defined in~\cite[Proposition~3.5]{center2025}, such that the following statements hold:
\begin{enumerate}
\item $h_{\omega}^{c}$ is a homeomorphism and is Lipschitz continuous.
\item for any $k\in \mathbb N$, if the map $\phi_{t_{0}}^{\omega}$ is of class $C^{k}$, then the manifold $\mathcal W_{\mathrm{loc}}^{c,\nu}(\omega)$ is of class $C^{k-1}$.
    
\item $\mathcal W_{\mathrm{loc}}^{c,\nu}(\omega)$ is $\bar{\phi}$-invariant, i.e., for any $n\in \mathbb N_0$, we have
\[\bar{\phi}_{nt_{0}}^{\omega}(\mathcal W_{\mathrm{loc}}^{c,\nu}(\omega))\subseteq \mathcal W_{\mathrm{loc}}^{c,\nu}(\theta_{nt_{0}}\omega).\] 
\item for any $z_{\omega}\in \mathcal W_{\mathrm{loc}}^{c,\nu}(\omega)$, there exists a sequence $\{z_{\theta_{-nt_{0}}\omega}\}_{n\in\mathbb N_0}$, such that, if for any negative integer $j$, we define $\bar{\phi}_{jt_{0}}^{\omega}(z_{\omega}):=z_{\theta_{jt_{0}}\omega}$, then for any integer $j$ and $k\in \mathbb N_0$,
    \[
    \bar{\phi}_{kt_{0}}^{\theta_{jt_{0}}\omega}(\bar{\phi}_{jt_{0}}^{\omega}(z_{\omega}))
    =\bar{\phi}_{(k+j)t_{0}}^{\omega}(z_{\omega}),
    \]  
and 
    \[
    \sup_{j}\exp(-\nu |j|)\|\bar{\phi}_{jt_{0}}^{\omega}(z_{\omega})-Y_{\theta_{jt_{0}}\omega}\|<\infty.
    \]
\end{enumerate}
\end{enumerate}  
\end{theorem}

\begin{proof}
By \cite[Theorem 3.14]{center2025}, it is sufficient to check that~\eqref{rela-GVR2023-1.1} and \eqref{eq-GVR2023-assum2.5} hold.
The argument proceeds along the same lines as for Theorem~\ref{thm:manifold-local-stable}, which completes the proof.
\end{proof}

\subsection{Finite-time Lipschitz estimate and continuous-time upgrade}
\label{subsec:manifold-continuous-time}
The invariant manifolds obtained above are initially constructed for the time-$t_0$ cocycle. We now establish a finite-time Lipschitz estimate on $[0,t_0]$ and use it to extend the discrete-time manifold properties to continuous time.

\begin{proposition}
\label{prop:manifold-tempered-finite-lipschitz}
Define 
$$L_{\mathrm{loc}}:\Omega\to \mathbb R,\qquad \omega\mapsto L_{\mathrm{loc}}(\omega)
:=
\sup_{0\le r\le t_0}\|\psi_r^\omega\|
+
f_{\mathrm{nl}}(\omega).$$
Then there exists a positive tempered random radius $\rho_L(\omega)$ such that, whenever
\[
\|y_i-Y_\omega\|\le\rho_L(\omega),
\qquad
i=1,2,
\]
one has
\begin{equation}
\label{eq:manifold-tempered-finite-lipschitz}
\sup_{0\le r\le t_0}
\|\phi_r^\omega(y_2)-\phi_r^\omega(y_1)\|
\le
L_{\mathrm{loc}}(\omega)\|y_2-y_1\|.
\end{equation}
Moreover, $L_{\mathrm{loc}}$ is tempered in the sense of Definition~\ref{def:sp-tempered-radius}.
\end{proposition}
\begin{proof}
Let
\[
y_\theta:=y_1+\theta(y_2-y_1)=: Y_\omega+z_\theta,
\qquad
0\le\theta\le1.
\]
By $\|y_i-Y_\omega\|\le \rho_{L}(\omega)$ for any $i=1,2$, we choose $\rho_L (\omega)\le 1$ so that $\|z_\theta\|\le 1$. Then
\begin{align*}
\|D_{y_\theta}\phi_r^\omega\|=& \|D_{Y_\omega+z_\theta}\phi_r^\omega\|\\
\le & \|D_{Y_\omega}\phi_r^\omega\|+\|D_{Y_\omega+z_\theta}\phi_r^\omega-D_{Y_\omega}\phi_r^\omega\|\\
\le & \|D_{Y_\omega}\phi_r^\omega\|+f_{\mathrm{nl}}(\omega)\|z_\theta\|\hspace{1cm}(\text{by }\eqref{eq:manifold-derivative-estimate-stationary})\\
\le & \|\psi_r^\omega\|+f_{\mathrm{nl}}(\omega).\hspace{1cm}(\text{by }\|z_\theta\|\le 1),
\end{align*}
and so
\begin{align*}
\sup_{0\le r\le t_0}\|\phi_r^\omega(y_2)-\phi_r^\omega(y_1)\|\le & \sup_{0\le r\le t_0}\sint_0^1\|D_{y_\theta}\phi_r^\omega(y_2 -y_1)\|d\theta \hspace{1cm}(\text{by the mean value formula})\\
\le & \sup_{0\le r\le t_0}\|D_{y_\theta}\phi_r^{\omega}\|\, \|y_2 -y_1\|\\
\le & (\sup_{0\le r\le t_0}\|\psi_r^\omega\|+f_{\mathrm{nl}}(\omega)) \|y_2 -y_1\|\\
=& L_{\mathrm{loc}}(\omega)\|y_2-y_1\|.
\end{align*}
Thus~\eqref{eq:manifold-tempered-finite-lipschitz} holds. It remains to prove temperedness of $L_{\mathrm{loc}}$. For the term $\sup_{0\le r\le t_0}\|\psi_r^\omega\|$, by Corollary~\ref{cor:lin-local-linear-control-integrable}, we get
\begin{equation*}
\sup_{0\le r\le t_0}\|\psi_r^\omega\|\le \exp\bigg(\sup_{0\le t\le t_0}\log^+\|\psi_t^\omega\|+\sup_{0\le t\le t_0}
\log^+\|(\psi_t^\omega)^{-1}\|\bigg)=\exp(M_{\mathrm{lin}}(\omega)), \quad \text{for fixed } t_0\in\mathbb{R}_{>0}.
\end{equation*}
Since $M_{\mathrm{lin}}(\omega)\in L^1(\Omega)$, its exponential is tempered.
The term $f_{\mathrm{nl}}$ is tempered by Proposition~\ref{lem:manifold-fnl-tempered}.
By the calculus of tempered variables, their sum is tempered.
Therefore 
\begin{align}~\label{eq:Lloc-tempered}
	\lim_{|n|\to\infty}
	\frac1{|n|}
	\log^+L_{\mathrm{loc}}(\theta_{nt_0}\omega)=0.
\end{align} It completes the proof. 
\end{proof}

Next we consider the property of stable manifold in the case of continuous time.

\begin{corollary}
\label{cor:manifold-continuous-time-stable}
Let $0<\nu<\nu_1$, and let $y\in \mathcal W_{\mathrm{loc}}^{s,\nu}(\omega)$.
Then, after possibly shrinking the stable radius,
\begin{align*}
\sup_{t\ge0}
\exp(\nu_1 t)
\|\phi_t^\omega(y)-Y_{\theta_t\omega}\|
<\infty.
\end{align*}
\end{corollary}
\begin{proof}
For any $t\in \mathbb R_{\ge 0}$, we write  
\[
t:=nt_0+r,
\qquad
n\in\mathbb N_0,
\qquad
0\le r<t_0.
\]
For $y$ in the local stable manifold and after shrinking the stable radius, both arguments remain in the finite-time Lipschitz neighborhood for all large $n$. Then
\begin{align*}
\|\phi_t^\omega(y)-Y_{\theta_t\omega}\|=&\ 
\|\phi_r^{\theta_{nt_0}\omega}
\bigl(
\phi_{nt_0}^\omega(y)
\bigr)
-
\phi_r^{\theta_{nt_0}\omega}
\bigl(
Y_{\theta_{nt_0}\omega}
\bigr)\|\hspace{1cm}(\text{by~\meqref{eq:mecoc} and \eqref{eq:stationary} })\\
\le &\ L_{\mathrm loc}(\theta_{nt_0}\omega)\|\phi_{nt_0}^\omega(y)-Y_{\theta_{nt_0}\omega}\|\hspace{1cm}(\text{by }\eqref{eq:manifold-tempered-finite-lipschitz})\\
\le &\ L_{\mathrm{loc}}(\theta_{nt_0}\omega)C(\omega,y)\exp(-\nu nt_0)\hspace{1cm}(\text{by }y\in \mathcal{W}_{\mathrm{loc}}^{s,\nu}(\omega))\\
\le &\ C_{\nu_1}(\omega,y)\exp(-\nu_1 (nt_0+r))\hspace{1cm}(\text{by~\eqref{eq:Lloc-tempered} })\\
\le &\ C_{\nu_1}(\omega,y)\exp(-\nu_1 t).
\end{align*}
Thus, \[\sup_{t\ge 0}\exp(\nu_1 t)\|\phi_{t}^{\omega}(y)-Y_{\theta_{t}\omega}\|<\infty.\]
This completes the proof.
\end{proof}

\begin{remark}
Under the conditions of Corollary~\ref{cor:manifold-continuous-time-stable}, for any $y\in \mathcal W_{\mathrm{loc}}^{s,\nu}(\omega)$ with $0<\nu \le \nu_1 <-\mu^{-}$, we have 
$y\in \mathcal W_{\mathrm{loc}}^{s,\nu_1}(\omega)$ by Theorem~\ref{thm:manifold-local-stable}. Thus
$\mathcal{W}_{\mathrm{loc}}^{s,\nu_1}(\omega)\subseteq \mathcal{W}_{\mathrm{loc}}^{s, \nu}(\omega).$
\end{remark}

We now consider the properties of the local stable manifold in continuous time $t\in\mathbb R_{\ge0}$.

\begin{theorem}~\label{thm-continu}
	For every $t\in \mathbb R_{\ge 0}$, the following statements hold:
	\begin{enumerate}
		\item~\label{prop-continu-exp}  the solution flow still converges to $Y_\omega$ at an exponential rate, i.e.,
		\[\limsup_{t\to\infty}\frac{1}{t}\log\|\phi_t^\omega(z)-Y_{\theta_t\omega}\|< \infty,\qquad \forall z\in \mathcal{W}_{\mathrm{loc}}^{s,\nu}(\omega).\]
		\item ~\label{prop-contin-stab}
		if $t$ is large enough, then for any $0<\nu_1<\nu$, 
		\[\phi_t^\omega(\mathcal{W}_{\mathrm{loc}}^{s,\nu}(\omega))\subseteq \mathcal{W}_{\mathrm{loc}}^{s,\nu_1}(\theta_t \omega).\]
	\end{enumerate}
\end{theorem}

\begin{proof}
\eqref{prop-continu-exp}. 
For any $t\in\mathbb R_{\ge 0}$, there exist $n\in\mathbb N_0$ and $0\le t_1< t_0$ such that $t=nt_0+t_1.$ For any $z\in \mathcal{W}_{\mathrm{loc}}^{s,\nu}(\omega)$, we have
\begin{align}\label{rem213}
&\ \sup_{0\le t_1 <t_0}\|\phi_{nt_0+t_1}^\omega(z)-\phi_{nt_0+t_1}^\omega(Y_\omega)\|\nonumber\\
=& \sup_{0\le t_1 <t_0}\|\phi_{t_1}^{\theta_{nt_0}\omega}(\phi_{nt_0}^\omega(z))-\phi_{t_1}^{\theta_{nt_0}\omega}(\phi_{nt_0}^{\omega}(Y_\omega))\|\nonumber\\
\le & \| \phi_{t_1}^{\theta_{nt_0}\omega}(\phi_{nt_0}^\omega(z))-\phi_{t_1}^{\theta_{nt_0}\omega}(\phi_{nt_0}^\omega(Y_\omega))\|_{D_{\bf X(\omega)}^\alpha([0,t_0])}\nonumber\\
\le & \|\phi_{nt_0}^\omega(z)-\phi_{nt_0}^{\omega}(Y_\omega)\|\exp(\widetilde{Q}_{\mathrm{lin}}(\|\phi_{nt_0}^{\omega}(z)\|,\|\phi_{nt_0}^\omega(Y_\omega)\|,
\|\mathbf Z(\omega)\|_{\mathcal{C}_{\bf X}^{\alpha}([0,T])},\|{\bf X(\omega)}\|_{\alpha,[0,T]})) \nonumber\\
&\hspace{1cm}(\text{by }\eqref{phi-y-z-contin})\\
\le & \|\phi_{nt_0}^\omega(z)- Y_{\theta_{nt_0}\omega}\|\cdot f(\omega),\nonumber
\end{align}
where $f:\Omega\to \mathbb{R}$ is tempered. Thus, 
\begin{align*}
&\ \limsup_{n\to \infty}\frac{1}{n}\log\big(\sup_{0\le t_1<t_0}(\|\phi_{nt_0+t_1}^\omega(z)-\phi_{nt_0+t_1}^\omega(Y_\omega)\|)\big)\\
\le &\ 
\limsup_{n\to \infty}\frac{1}{n}\big(\log(\|\phi_{nt_0}^\omega(z)-Y_{\theta_{nt_0}\omega}\|\cdot f(\omega))\big)\\
\le &\  C\cdot t_0\mu^{-},
\end{align*}
where $C$ is a constant depending on $f$, and so
\[\limsup_{t\to\infty}\frac{1}{t}\log\|\phi_t^\omega(z)-Y_{\theta_t\omega}\|< \infty.\]

\eqref{prop-contin-stab}.  We take $t\in \mathbb R_{\ge 0}$ with $t=mt_0+t_1$, where $0\le t_1<t_0$ and $0<\nu_1<\nu$. Then
\begin{align*}
	&\ \sup_{n\ge 0}\exp(nt_0 \nu_1)\|\phi_{nt_0}^{\theta_t\omega}\phi_t^\omega(z)-Y_{\theta_{nt_0+t}\omega}\|\\
	=&\ \sup_{n\ge 0}\exp(nt_0 \nu_1)\|\phi_{(m+n)t_0+t_1}^{\omega}(z)-\phi_{(m+n)t_0+t_1}^{\omega}(Y_\omega)\|\\
	\le &\ \sup_{n\ge 0}\exp(nt_0\nu_1)\|\phi_{(m+n)t_0}^\omega(z)-Y_{\theta_{(m+n)t_0}\omega}\|\exp(\widetilde{Q}_{\mathrm{lin}}(\|\phi_{nt_0}^{\omega}(z)\|,\|\phi_{nt_0}^\omega(Y_\omega)\|,
	\|\mathbf Z(\omega)\|_{\mathcal{C}_{\bf X}^{\alpha}([0,T])},\|{\bf X(\omega)}\|_{\alpha,[0,T]}))\\
	&\ \hspace{1cm}(\text{by }\eqref{rem213})\\
	\le &\ \sup_{n\ge 0}\exp(nt_0 \nu_1)\frac{1}{\exp((m+n)t_0\nu)}\exp((m+n)t_0\nu)\|\phi_{(m+n)t_0}^{\omega}(z)-Y_{\theta_{(m+n)t_0}\omega}\|\\
	&\ \times \exp(\widetilde{Q}_{\mathrm{lin}}(\|\phi_{nt_0}^{\omega}(z)\|,\|\phi_{nt_0}^\omega(Y_\omega)\|,
	\|\mathbf Z(\omega)\|_{\mathcal{C}_{\bf X}^{\alpha}([0,T])},\|{\bf X(\omega)}\|_{\alpha,[0,T]}))\\
	\le &\ \sup_{k\ge 0}\exp(kt_0\nu)\|\phi_{kt_0}^\omega(z)-Y_{\theta_{kt_0}\omega}\|\cdot \sup_{n\ge 0}\exp(nt_0\nu_1-(m+n)t_0\nu)\\
	&\ \times \exp(\widetilde{Q}_{\mathrm{lin}}(\|\phi_{nt_0}^{\omega}(z)\|,\|\phi_{nt_0}^\omega(Y_\omega)\|,
	\|\mathbf Z(\omega)\|_{\mathcal{C}_{\bf X}^{\alpha}([0,T])},\|{\bf X(\omega)}\|_{\alpha,[0,T]}))\\
	\triangleq &\ \sup_{k\ge 0}\exp(kt_0\nu)\|\phi_{kt_0}^\omega(z)-Y_{\theta_{kt_0}\omega}\|\times I(\omega,t,z).
\end{align*}
On the one hand, owing to $z\in \mathcal{W}_{\mathrm{loc}}^{s,\nu}(\omega)$, we have
\[C(z,\omega):=\sup_{k\ge 0}\exp(kt_0\nu)\|\phi_{kt_0}^\omega(z)-Y_{\theta_{kt_0}\omega}\|<\infty.\]
On the other hand, by the definition of $I(\omega,t,z)$,
\begin{align*}
	&\ I(\omega,t,z)\\
	=&\ \sup_{n\ge 0}\exp(nt_0\nu_1-(m+n)t_0\nu) \exp(\widetilde{Q}_{\mathrm{lin}}(\|\phi_{nt_0}^{\omega}(z)\|,\|\phi_{nt_0}^\omega(Y_\omega)\|,
	\|\mathbf Z(\omega)\|_{\mathcal{C}_{\bf X}^{\alpha}([0,T])},\|{\bf X(\omega)}\|_{\alpha,[0,T]}))\\
	=&\ \sup_{n\ge 0}\bigg( \exp(-mt_0 \nu)\exp(nt_0(\nu_1-\nu)) \bigg) \exp(\widetilde{Q}_{\mathrm{lin}}(\|\phi_{nt_0}^{\omega}(z)\|,\|\phi_{nt_0}^\omega(Y_\omega)\|,
	\|\mathbf Z(\omega)\|_{\mathcal{C}_{\bf X}^{\alpha}([0,T])},\|{\bf X(\omega)}\|_{\alpha,[0,T]})).
\end{align*}
Notice that $I(\omega,t,z)$ can be made arbitrarily small for sufficiently large $n$. 
Thus, for $0<\nu_1<\nu$,  
\begin{align*}
	\sup_{n\ge 0}\exp(nt_0 \nu_1)\|\phi_{nt_0}^{\theta_t\omega}\phi_t^\omega(z)-Y_{\theta_{nt_0+t}\omega}\|
	\le C(z,\omega)\times I(\omega,t,z) <\infty.
\end{align*}
That is, $\phi_t^\omega(z)\in \mathcal{W}_{\mathrm{loc}}^{s,\nu_1}(\theta_{t}\omega)$. So we get $\phi_t^\omega(\mathcal{W}_{\mathrm{loc}}^{s,\nu}(\omega))\subseteq \mathcal{W}_{\mathrm{loc}}^{s,\nu_1}(\theta_t \omega)$, as required. 
\end{proof}

The corresponding continuous-time statements for local unstable manifolds follow by the same argument applied to backward orbits.

\begin{remark}
For local unstable manifolds, we obtain the following statements, analogous to those in Theorem~\ref{thm-continu}.
\begin{enumerate}
\item For $z_\omega\in \mathcal{W}_{\mathrm{loc}}^{u,\nu}(\omega)$, if $t\in\mathbb R_{>0}$, then
 \[\limsup_{t\to \infty}\frac{1}{t}\log\|z_{\theta_{-t}\omega}-Y_{\theta_{-t}\omega}\|<\infty.\]
\item Moreover, if $0< \nu_1<\nu$, then for sufficiently large $t$, one has
$
 \mathcal{W}_{\mathrm{loc}}^{u,\nu}(\omega)\subseteq \phi_t^{\theta_{-t}\omega}(\mathcal{W}_{\mathrm{loc}}^{u,\nu_1}(\theta_{-t}\omega)).
 $
\end{enumerate}
\end{remark}

\vskip 0.2in

\noindent
{\bf Acknowledgments.} Xing Gao is supported by the National Natural Science Foundation of China (12571019), the Natural Science Foundation of Gansu Province (25JRRA644) and Innovative Fundamental Research Group Project of Gansu Province (23JRRA684).

\noindent
{\bf Declaration of interests. } The authors have no conflicts of interest to disclose.

\noindent
{\bf Data availability. } Data sharing is not applicable as no new data were created or analyzed.

\end{document}